\documentclass[reqno,12pt]{amsproc}
\usepackage{ucs}
\usepackage[english]{babel}
\usepackage{amssymb,amsmath}
\usepackage{indentfirst}
\usepackage[usenames]{color}
\usepackage{graphicx}
\usepackage[left=2cm,right=2cm,top=1.5cm,bottom=1.5cm,bindingoffset=0cm]{geometry}
\usepackage{amsthm}
\usepackage{amsfonts}
\usepackage{graphicx}
\graphicspath{{pictures/}}
\DeclareGraphicsExtensions{.pdf,.png,.jpg}
\usepackage{wrapfig}

\theoremstyle{plain}  
\newtheorem*{acknowledgements}{Acknowledgements}
\newtheorem{theorem}{\bf Theorem}
\newtheorem{lemma}{\bf Lemma}

\newtheorem{oldthm}{\bf Theorem}

\newtheorem{remark}{\bf Remark}

\renewcommand{\a}{\alpha}

\usepackage{xcolor}
\usepackage{hyperref}

\usepackage{float}

\hypersetup{pdfstartview=FitH, linkcolor=linkcolor,urlcolor=urlcolor, colorlinks=true, linkcolor=blue, citecolor=blue}

\begin{document}

\title{Two-dimensional Hardy-Littlewood theorem \\for functions with general monotone Fourier coefficients}
\author{Kristina Oganesyan}
\address{Lomonosov Moscow State University, Moscow Center for Fundamental and Applied Mathematics, Centre de Recerca Matem\`atica, Universitat Aut\`onoma de Barcelona}
\email{oganchris@gmail.com}
\thanks{The work was supported by the Moebius Contest Foundation for Young Scientists and the Foundation for the Advancement of Theoretical Physics and Mathematics $``$BASIS$"$ (grant no 19-8-2-28-1).}
\date{}

\begin{abstract}
We prove the Hardy-Littlewood theorem in two dimensions for functions whose Fourier coefficients obey general monotonicity conditions and, importantly, are not necessarily positive. The sharpness of the result is given by a counterexample, which shows that if one slightly extends the considered class of coefficients, the Hardy-Littlewood relation fails.
\end{abstract}
\maketitle

\section{Introduction}
Establishing interconnections between integrability of functions and summability of their Fourier coefficients is the problem which occupies a special place in harmonic analysis. The celebrated Parseval's identity enables us to reduce a wide class of problems concerning functions to those concerning their Fourier series, and vice versa. Although we do not have such equalities for the spaces $L_p,\;p\neq 2,$ we can still obtain equivalences of norms of functions and norms of their Fourier series if we impose some additional requirements. Results of this kind are important, in the first place, due to the fact that once such a relation is found, one becomes free to choose if it is handy to deal with functions or with coefficients in this or that case, as if having Parseval's identity (see, e.g. \cite[Chs. 4--6, 12--13]{DoT} and \cite[Sec. 7]{GT} for applications). The following result by Paley \cite{P} can be considered the starting point for the research in this direction.

\begin{oldthm}[Paley, 1931]\label{paley} Let $\{\phi_n(x)\}$ be an orthonormal system on $[a,b]$ with $|\phi_n(x)|\leq M$ for all $x\in[a,b]$ and $n\in\mathbb{N}$. Then
\\
a) If $p\in(1,2],$ then for any $f\in L_p(a,b)$ with Fourier coefficients $\{c_n\}$ there holds
\begin{align}\label{p1}
\sum_{n=1}^{\infty}|c_n|^pn^{p-2}\lesssim_{p,M}\|f\|_p^p.
\end{align}
b) If $p\in[2,\infty),$ then, for any sequence $\{c_n\}$ with $\sum_{n=1}^{\infty}|c_n|^pn^{p-2}<\infty$, there exists a function $f\in L_p(a,b)$ that has $\{c_n\}$ as its Fourier coefficients and 
\begin{align}\label{p2}
\sum_{n=1}^{\infty}|c_n|^pn^{p-2}\gtrsim_{p,M}\|f\|_p^p.
\end{align}
\end{oldthm}
Throughout the paper, for two functions $f$ and $g$, the relation $f\gtrsim g$ (or $g\lesssim f$) means that there exists a constant $C$ such that $f(x)\geq Cg(x)$ for all $x$, and the relation $f\asymp g$ is equivalent to $f\gtrsim g\gtrsim f$ (if we write $f\gtrsim_a g$, this means that the corresponding constant is allowed to depend on $a$). From now on, we discuss Fourier series only with respect to the trigonometric system.

The ranges of $p$ in Theorem \ref{paley} are sharp, therefore to have both \eqref{p1} and \eqref{p2} true for all $p\in (1,\infty)$, one has to impose some additional requirements. Hardy and Littlewood \cite{HL} showed that if we restrict ourselves to sine or cosine series with monotone tending to zero coefficients, then both relations \eqref{p1} and \eqref{p2} hold for all $p\in(1,\infty)$. In this regard, a natural question to ask was: how much can we release the requirement of monotonicity to have 
\begin{align}\label{p}
\sum_{n=1}^{\infty}|c_n|^pn^{p-2}\asymp_p\|f\|_p^p
\end{align}
still true? This question in turn motivated creation of various extentions of the class of monotone sequences satisfying \eqref{p}. One of these classes, the so-called general monotone or just $GM$ class \cite[Th. 4.2]{T-1}, consists of all sequences $\{a_n\}$ fulfilling the condition
\begin{align}\label{gm}
\sum_{k=n}^{2n}|a_{k}-a_{k+1}|\lesssim |a_n|
\end{align}
for all $n$. Thus, now we dropped not only the monotonicity condition but even the basic requirement of positivity, keeping though some regularity of our sequences. One can see that $GM$ class can yet be generalized (see \cite[Th. 6.2(B)]{T-2} and \cite[Th. 1]{YZZ-1}) by putting a mean value on the right-hand side of \eqref{gm} instead of $|a_n|$ as follows:
\begin{align}\label{gm0}
\sum_{k=n}^{2n}|a_k - a_{k+1}| \lesssim \sum_{k=\frac{n}{\lambda}}^{\lambda n}\frac{|a_k|}{k}
\end{align}
for some $\lambda>1$ (see also \cite{FTZ} for some properties of such sequences). Note that these classes and several other ones, defined as \eqref{gm} but with some other majorants on the right-hand side, in different sources can be also called $GM$. For a comprehensive survey on the concept of general monotonicity, we refer the reader to \cite{LT}. 

One more direction of extending the obtained results (see \cite{B,H,AW,YZZ-1}) is proving them for weighted spaces. Define weighted Lebesgue spaces $L_{w(p,q)}^q,\;p,q\in(0,\infty],$ on $[-\pi,\pi]$, as the set of all measurable functions $f$ with finite norm
\begin{equation*}
\|f\|_{L_{w(p,q)}^q} :=
\begin{cases}
\Big(\int\limits_{-\pi}^{\pi} |t|^{\frac{q}{p}-1}|f(t)|^q dt\Big)^{\frac{1}{q}},\quad \text{if} \; 0 < p,q < \infty, \\
\operatorname*{ess\,sup}\limits_{t\in [-\pi,\pi]}\  |t^{\frac{1}{p}} f(t)|,\quad \text{if} \;  0 < p \leq \infty,\; q = \infty.
\end{cases}
\end{equation*}
The discrete weighted Lebesgue space $l_{w(p,q)}^q$ is to be defined in the same way. 

Now, a weighted version of relation \eqref{p} is
\begin{equation}\label{pp}
\|\{c_n\}\|_{l_{w(p',q)}^q}^q:=\sum_{n=1}^{\infty}|c_n|^qn^{\frac{q}{p'}-1}\asymp \|f\|_{L_{w(p,q)}^q}^q,
\end{equation} 
where $p'$ stands for the conjugate to $p$, that is, $1/p+1/p'=1$. Note that if we put $q=p$, we get the standard Hardy-Littlewood relation \eqref{p}. From now on, writing Hardy-Littlewood type relations we will omit the dependence on $p$ of the corresponding constants, so this dependence will be taken for granted.  The following theorem for weighted Lebesgue spaces was obtained by Sagher \cite{S}.
\begin{oldthm}[Sagher, 1976]\label{sagh}
If the sequences $\{a_n\}$ and $\{b_n\}$ are monotone and vanishing at infinity and the function $f$ has the Fourier series
\begin{equation*}
\frac{a_0}{2}+\sum_{n=1}^{\infty} (a_n \cos nx + b_n \sin nx),
\end{equation*}
then for $p\in (1,\infty),\;q\in[1,\infty],$ there holds
\begin{equation*}
\|f\|_{L_{w(p,q)}^q} \asymp \|\{a_n\}\|_{l_{w(p',q)}^q} + \|\{b_n\}\|_{l_{w(p',q)}^q}.
\end{equation*}
\end{oldthm}
It turns out that the same holds if we release the monotonicity condition in the theorem above to \eqref{gm0}, thus withdrawing the requirement of positivity. This result, along with the similar statement proved for Lorentz spaces, was given by Dyachenko, Mukanov and Tikhonov \cite{DMT}.

So, in the one-dimensional case we have quite a complete picture.

The whole scenario becomes more complicated if we step out from the one-dimensional setting to the multidimensional one, and the first question we face is to determine what we should mean by monotonicity if we deal with multiple sequences. The usual one-dimensional monotonicity is characterized by the inequalities $a_n\geq a_{n+1}$, or equivalently, $\Delta a_n:=a_n-a_{n+1}\geq 0$. These two ways of writing the same property give rise to the following fundamentally different multidimensional monotonicity concepts. Our focus will be on the two-dimensional case.

\subsection*{Monotonicity in each variable.} Likewise $a_n\geq a_{n+1}$ in one dimension, we can require  coordinatewise monotonicity, that is, in two-dimensional case the condition will be 
\begin{align}\label{dcond}
a_{mn}\leq a_{m'n'},\quad \text{for all}\;\;m\geq m',\;n\geq n'.
\end{align}
It turns out, however, that for such sequence the Hardy-Littlewood relation \eqref{p} does not hold for some values of $p>1$, namely, we have the following result proved by Dyachenko \cite{D-1,D-3}.

\begin{oldthm}[Dyachenko, 1986] a) \cite[Th. 1]{D-1} If $\{a_{mn}\}_{m,n=1}^{\infty}$ satisfying \eqref{dcond} and
\begin{align}\label{cond1}
a_{mn}\to 0,\quad \text{as}\;m+n\to\infty,
\end{align}
is the sequence of the Fourier coefficients with respect to one of the orthonormal systems $\{e^{inx}e^{imy}\}_{m,n=1}^{\infty}$, $\{\sin nx\sin my\}_{m,n=1}^{\infty}$, and $\{\cos nx\cos ny\}_{m,n=1}^{\infty}$, of a function $f$, then for any $p\in(1,\infty)$,
\begin{align*}
\sum_{m,n=1}^{\infty}a_{mn}^p(mn)^{p-2}\lesssim \|f\|_p^p.
\end{align*}
b) \cite[Cor. 2]{D-3} Let $p> 4/3$ and the sequence $\{a_{mn}\}$ satisfy \eqref{dcond} and $\sum_{m,n=1}^{\infty}a_{mn}^p(mn)^{p-2}<\infty$ (therefore, \eqref{cond1} as well). Then, for any of the systems above, there exists a function $f$ having $\{a_{mn}\}$ as its Fourier coefficients and satisfying
\begin{align}\label{bpart}
\sum_{m,n=1}^{\infty}a_{mn}^p(mn)^{p-2}\gtrsim \|f\|_p^p.
\end{align}
c) \cite[Ths. 8, 8']{D-1} For $p\in(1,4/3)$, there exists a sequence $\{a_{mn}\}$ satisfying \eqref{dcond} and \eqref{cond1} with $\sum_{m,n=1}^{\infty}a_{mn}^p(mn)^{p-2}<\infty$ such that the corresponding trigonometric series diverges by squares almost everywhere on $(0,2\pi)^2$.
\end{oldthm}
Note that it was shown by Fefferman \cite{F} that for any $p>1$ and any $f\in L_p(0,2\pi)^2$, the Fourier series of $f$ converges by squares almost everywhere on $(0,2\pi)^2$, thus, the third part of the theorem means that \eqref{bpart} is no longer true for $p\in(1,4/3)$. We also remark that in general $d$-dimensional case the critical value is $2d/(d+1)$ (see \cite[Th. 1, Th. 4]{D-2} and \cite[Cor. 2]{D-3}).

\subsection*{Monotonicity by Hardy.} The next approach to the multiple concept of monotonicity is to consider the monotonicity in the so-called sense of Hardy (or Hardy-Krause, see \cite{Ha} and \cite{K}, where this concept initially arises). That is, to introduce the following differences
\begin{align*}
&\Delta^{10}a_{mn}:=a_{mn}-a_{m+1,n},\qquad\Delta^{01}a_{mn}:=a_{mn}-a_{m,n+1},\\
&\Delta^{11}a_{mn}:=\Delta^{01}(\Delta^{10}a_{mn})=\Delta^{10}(\Delta^{01}a_{mn})=a_{mn}-a_{m+1,n}-a_{m,n+1}+a_{m+1,n+1},
\end{align*}
and recalling the one-dimentional condition $\Delta a_n\geq 0$, generalize it in the following way
\begin{align}\label{mcond}
\Delta^{11}a_{mn}\geq 0\quad\text{for all}\; m,n.
\end{align}
Note that under the natural requirement \eqref{cond1}, condition \eqref{mcond} implies
\begin{align*}
a_{mn}\geq 0,\quad \Delta^{10}a_{mn}\geq 0,\quad \Delta^{01}a_{mn}\geq 0.
\end{align*} 
Here comes the result obtained by M\'oricz \cite[Th. 1,2, Cor. 1]{M}.

\begin{oldthm}[M\'oricz, 1990]\label{moricz} Let $p\geq 1$ and the sequence $\{a_{mn}\}$ satisfy \eqref{cond1} and \eqref{mcond}.

a) If $\sum_{m,n=1}^{\infty}a_{mn}^p(mn)^{p-2}<\infty$, then the double sine or cosine series with coefficients $\{a_{mn}\}$ is the Fourier series of its sum $f$ and 
\begin{align*}
\sum_{m,n=1}^{\infty}a_{mn}^p(mn)^{p-2}\gtrsim \|f\|_p^p. 
\end{align*}

b) If $\{a_{mn}\}$ is the sequence of double sine or cosine Fourier coefficients of $f\in L_p$, then
\begin{align*}
\sum_{m,n=1}^{\infty}a_{mn}^p(mn)^{p-2}\lesssim \|f\|_p^p. 
\end{align*}
\end{oldthm}
The reader can find Theorem \ref{moricz} proved for Vilenkin systems (and hence for the Walsh system) in \cite[Sec. 6.3]{W-1} and \cite{W-2}.

Condition \eqref{mcond} is quite restrictive and one of the closest generalizations of it in, say, $GM$ spirit is the following one
\begin{align*}
\sum_{m=k}^{\infty}\sum_{n=l}^{\infty}|\Delta^{11}a_{mn}|\lesssim |a_{mn}|.
\end{align*}
Note that if the sequence satisfies \eqref{mcond}, then the left-hand side above becomes just equal to $a_{mn}$. The next result \cite[Th. 6B]{DT-0} (see \cite{DT-1} for the proof) extends the one of M\'oricz.
\begin{oldthm}[Dyachenko, Tikhonov, 2007]\label{4} If a nonnegative sequence $\{a_{mn}\}$ satisfy \eqref{cond1} and the so-called $GM^2$ condition
\begin{align}\label{dtcond}
\sum_{m=k}^{\infty}\sum_{n=l}^{\infty}|\Delta^{11}a_{mn}|\lesssim |a_{kl}|+\sum_{m=k}^{\infty}\frac{|a_{ml}|}{m}+\sum_{n=l}^{\infty}\frac{|a_{kn}|}{n}+\sum_{m=k}^{\infty}\sum_{n=l}^{\infty}\frac{|a_{mn}|}{mn},
\end{align}
then the corresponding double sine, cosine, or exponential series converges everywhere on $(0,2\pi)^2$ and is the Fourier series of its sum. Besides, for any $p\in (1,\infty)$,
\begin{align*}
\sum_{m,n=1}^{\infty}a_{mn}^p(mn)^{p-2}\asymp \|f\|_p^p.
\end{align*}
\end{oldthm}
It is worth mentioning that the $\gtrsim$ part was proved without assuming $a_{mn}\geq 0$, moreover, it was shown that if $\sum_{m=k}^{\infty}\sum_{n=l}^{\infty}|\Delta^{11}a_{mn}|\lesssim \beta_{kl}$, then $\sum_{m,n=1}^{\infty}\beta_{mn}^p(mn)^{p-2}\gtrsim \|f\|_p^p.$ However, in the proof of the counterpart the requirement of nonnegativity plays a crucial role. It was noted in \cite[Th. 4.1]{DT-11} that following the lines of this proof one can adapt it for a more general class of sequences for which the right-hand side of \eqref{dtcond} is replaces by $\sum_{m=\lceil k/\lambda\rceil}^{\infty}\sum_{n=\lceil l/\lambda\rceil}^{\infty}|a_{mn}|/mn,\;\lambda>1$. 

Further, it was shown \cite{YZZ-2} that some other $GM$ type nonnegative sequences happen to obey the two-sided Hardy-Littlewood relation. We present the result from \cite{YZZ-2} for weighted spaces.

\begin{oldthm}[Yu, Zhou, Zhou, 2012] Let $\{a_{mn}\}$ be a nonnegative sequence satisfying \eqref{cond1} and the following $GM$ type conditions
\begin{align*}
\sum_{m=k}^{2k}|\Delta a_{ml}|&\lesssim \sum_{m=\lfloor \lambda^{-1}k\rfloor}^{\lfloor \lambda k\rfloor}\frac{|a_{ml}|}{m},\qquad\sum_{n=l}^{2l}|\Delta a_{kn}|\lesssim \sum_{n=\lfloor \lambda^{-1}l\rfloor}^{\lfloor \lambda l\rfloor}\frac{|a_{kn}|}{n},\\
&\sum_{m=k}^{2k}\sum_{n=l}^{2l}|\Delta a_{mn}|\lesssim \sum_{m=\lfloor \lambda^{-1}k\rfloor}^{\lfloor \lambda k\rfloor}\sum_{n=\lfloor \lambda^{-1}l\rfloor}^{\lfloor \lambda l\rfloor}\frac{|a_{mn}|}{mn}
\end{align*}
for some $\lambda\geq 2$, and let $f(x,y):=\sum_{m,n=1}^{\infty}a_{mn}\sin mx\sin ny$. Then, for any $p\in[1,\infty)$, for any function $\phi\in \Phi$ with either $\phi^{-\frac{1}{p-1}}\in L$ if $p>1$, or $\phi^{-1}\in L_{\infty}$, if $p=1$, we have
\begin{align*}
\phi|f|^p\in L\Leftrightarrow \sum_{m,n=1}^{\infty}a_{mn}^p\phi(1/m,1/n)(mn)^{p-2}<\infty.
\end{align*}
\end{oldthm}
In the above result $\Phi$ stands for some class of power-like positive functions, which we are not going to specify here. A similar result with a more general $GM$ type positive sequences and some other (not comparable) class of power-like functions was obtained in \cite{DS}.

The main purpose of this work is to show that for some kinds of double $GM$ sequences we can prove the Hardy-Littlewood theorem without restricting ourselves only to positive sequences. We present two $GM$ type classes for which the two-sided Hardy-Littlewood inequality holds true. 

We write that $\{a_{mn}\}\in GM^c_1$ if it satisfies \eqref{cond1} and
\begin{align}\label{cond2}
\sum_{m=k}^{2k}\sum_{n=l}^{\infty}|\Delta^{11}a_{mn}|+\sum_{m=k}^{\infty}\sum_{n=l}^{2l}|\Delta^{11}a_{mn}|\leq C|a_{kl}|,
\end{align}
and $\{a_{mn}\}\in GM^c_2$, if it satisfies \eqref{cond1} and
\begin{align}\label{cond22}
\sum_{m=k}^{2k}\sum_{n=l}^{\infty}|\Delta^{11}a_{mn}|+\sum_{m=k}^{\infty}\sum_{n=l}^{2l}|\Delta^{11}a_{mn}|\leq C|a_{2k,l}|,
\end{align}
for all $k,l\in\mathbb{N}$ and some constant $C$ depending only on the sequence $\{a_{mn}\}$. We remark that the letter $c$ in $GM^c$ comes from the word $``$corner$"$, since a set of the kind $[k,2k]\times[l,\infty)\cup [k,\infty)\times[l,2l]$ generates a corner on the plane. Note that $GM^c_1$ sequences obey the one-dimensional $GM$ conditions \eqref{gm} in each variable (see \eqref{cond1'} in the proof of Lemma \ref{le}), while $GM^c_2$ in one variable satisfy \eqref{gm}, and in another one, the $``$backward$"$ $GM$ condition.

Note that for $[-\pi,\pi]^2$ the $L_{w(p,q)}^q$-norms take the form
\begin{equation*}
\|f\|_{L_{w(p,q)}^q} :=
\begin{cases}
\Big(\int\limits_{-\pi}^{\pi}\int\limits_{-\pi}^{\pi} |ts|^{\frac{q}{p}-1}|f(t,s)|^q dt\;ds\Big)^{\frac{1}{q}},\quad \text{if} \; 0 < p,q < \infty, \\
\operatorname*{ess\,sup}\limits_{(t,s)\in [-\pi,\pi]^2}\ |(ts)^{\frac{1}{p}} f(t,s)|,\quad \text{if} \;  0 < p \leq \infty,\; q = \infty.
\end{cases}
\end{equation*}

From now on, for convenience, we adopt the following notation: using that $(\sin x)^{(1)}=(\sin x)'=\cos x$ and $(\sin x)^{(0)}=\sin x$, we will write a two-dimensional trigonometric series as
\begin{align*}
\sum_{i,j=0}^1\sum_{m,n=0}^{\infty}a_{mn}^{ij}\sin^{(i)} mx\sin^{(j)} ny
\end{align*}
and we will say that $\{a_{mn}^{ij}\}_{m,n=1}^{\infty},\;i,j=0,1,$ is the sequence of its coefficients.

The main result of the paper is the following.

\begin{theorem}\label{m}
Let $p\in (1,\infty),\;q\in [1,\infty],$ and let each of the sequences $\{a_{mn}^{ij}\}_{m,n=1}^{\infty},\;i,j=0,1,$ belong either to $GM^c_1$ or to $GM^c_2$.

a) If $\{a_{mn}^{ij}\}_{m,n=1}^{\infty},\;i,j=0,1,$ is the sequence of Fourier coefficients of $f\in L(-\pi,\pi)$, then
\begin{align*}
\|f\|_{L_{w(p,q)}^q}^q \gtrsim \sum_{i,j=0}^1\sum_{m,n=1}^{\infty}|a_{mn}^{ij}|^q(mn)^{\frac{q}{p'}-1}.
\end{align*}

b) If $\sum_{i,j=0}^1\sum_{m,n=1}^{\infty}|a_{mn}^{ij}|^q(mn)^{\frac{q}{p'}-1}<\infty$, then the corresponding trigonometric series converges everywhere on $(0,2\pi)^2$ and is the Fourier series of its sum, moreover,
\begin{align*}
\|f\|_{L_{w(p,q)}^q}^q \lesssim\sum_{i,j=0}^1\sum_{m,n=1}^{\infty}|a_{mn}^{ij}|^q(mn)^{\frac{q}{p'}-1}.
\end{align*}

\end{theorem}
Sharpness of Theorem \ref{m} for $GM^c_2$ sequences is provided by a counterexample in Theorem \ref{re}, which shows that if we restrict the sum on the left-hand side of \eqref{cond22} to the rectangle (that is, to the intersection and not the union of the two corresponding strips), which is one of the most natural generalizations of the left-hand side of the $GM$ condition \eqref{gm}, then the $\gtrsim$ part fails for $p>2$ and $q\geq p$. %2p/(p-2)$ (note that if $q=p$ this is equivalent to the condition $p>4$).

\section{Proof of the Hardy-Littlewood theorem for $GM^c$ sequences}

For a sequence $\{a_{mn}\}_{m,n=1}^{\infty}$, we define
\begin{align*}
A_{mn}:=\max_{(k,l)\in Q_{m,n}}|a_{kl}|:=\max_{(k,l)\in [2^m,2^{m+1}]\times [2^{n},2^{n+1}]}|a_{kl}|.
\end{align*}

\begin{lemma}\label{le} a) For any sequence $\{a_{kl}\}_{k,l=1}^{\infty}\in GM^c_1$, there exist $c,v>0$ such that for any $(m,n)$ with $A_{m-1,n-1}\leq TA_{m,n}$ there exist a rectangle $Q'_{m-1,n-1}\subset Q_{m-1,n-1}$ of size $2^{m-v}\times 2^{n-v}$ satisfying 
\begin{align*}
\bigg|\sum_{k,l\in\; Q'_{m-1,n-1}}a_{kl}\bigg|> c2^{m+n}A_{mn},
\end{align*}
where $c$ and $v$ depend only on $C$ and $T$.

b) For any sequence $\{a_{kl}\}_{k,l=1}^{\infty}\in GM^c_2,$ there exist $c,v>0$ such that for any $(m,n)$ with $A_{m+1,n-1}\leq TA_{m,n}$ there exist a rectangle $Q'_{m+1,n-1}\subset Q_{m+1,n-1}$ of size $2^{m-v}\times 2^{n-v}$ satisfying 
\begin{align*}
\bigg|\sum_{k,l\in\; Q'_{m+1,n-1}}a_{kl}\bigg|> c2^{m+n}A_{mn},
\end{align*}
where $c$ and $v$ depend only on $C$ and $T$.
\end{lemma}

\begin{proof}
Note that \eqref{cond1} and \eqref{cond2} imply that
\begin{align}\label{cond1'}
\sum_{m=k}^{2k}|\Delta^{10}a_{mt}|+\sum_{n=l}^{2l}|\Delta^{01}a_{sn}|\leq C|a_{k,l}|
\end{align}
for any $k,l\in\mathbb{N}$ and $(s,t)\in [k,2k]\times [l,2l]$. Similarly, \eqref{cond1} along with \eqref{cond22} imply \eqref{cond1'} with $a_{2k,l}$ instead of $a_{k,l}$ on the right-hand side. In particular, \eqref{cond1'} yields that
\begin{align*}
|a_{s,t}|-|a_{k,l}|=|a_{s,t}|-|a_{k,t}|+|a_{k,t}|-|a_{2k,l}|\leq C|a_{k,l}|,
\end{align*}
so
\begin{align*}
|a_{s,t}|\leq (C+1)|a_{k,l}|\leq (C+1)^2|a_{s't'}|
\end{align*}
for any $(s',t')\in [0.5k,k]\times [0.5l,l]$. Considering $k=2^m,\;l=2^n,$ we get for any $(s,t)\in Q_{m-1,n-1}$
\begin{align}\label{ineq}
|a_{st}|\geq (C+1)^{-2}A_{mn}=:\a A_{mn}.
\end{align}

For conditions \eqref{cond1} and \eqref{cond22}, the same arguments give
\begin{align*}
|a_{s,t}|-|a_{2k,l}|=|a_{s,t}|-|a_{2k,t}|+|a_{2k,t}|-|a_{2k,l}|\leq C|a_{2k,l}|,
\end{align*}
and
\begin{align*}
|a_{s,t}|\leq (C+1)|a_{2k,l}|\leq (C+1)^2|a_{s't'}|
\end{align*}
for any $(s',t')\in [2k,4k]\times [0.5l,l]$. Once more, considering $k=2^m,\;l=2^n,$ we get \eqref{ineq} for $(s,t)\in Q_{m+1,n-1}$ instead of $Q_{m-1,n-1}$.

Thus, any sequence $\{a_{kl}\}\in GM^c_1$ satisfies $|a_{kl}|\leq (C+1)|a_{k'l'}|$ for $(k',l')\in [0.5k,k]\times [0.5l,l]$ as well as any $\{a_{kl}\}\in GM^c_2$ does for $(k',l')\in [k,2k]\times [0.5l,l]$.

In Lemma \ref{le}a), due to condition \eqref{cond1'} and inequality \eqref{ineq}, for any $(k,l)\in Q_{m-1,n-1}$, each one of the sequences $a_{2^{m-1},l},a_{2^{m-1}+1,l},...,a_{2^m,l}$ and $a_{k,2^{n-1}},a_{k,2^{n-1}+1},...,a_{k,2^n}$ can have at most 
\begin{align}\label{beta}
\frac{C\max\limits_{(k,l)\in Q_{m-1,n-1}}|a_{kl}|}{2\a A_{mn}}= \frac{C A_{m-1,n-1}}{2\a A_{mn}}\leq \frac{CT}{2\a}=:b
\end{align}
changes of sign.

The same holds for $Q_{m+1,n-1}$ in place of $Q_{m-1,n-1}$ in Lemma \ref{le}b). 

Focus now on Lemma \ref{le}a). Consider the rectangle $R:=Q_{m-1,n-1}=[2^{m-1},2^{m}]\times [2^{n-1},2^{n}]$ on the plane and draw all the segments $[(k,l),(k+1,l)]$ such that $a_{k,l-1}$ and $a_{k,l}$ have different signs and all the segments $[(k,l),(k,l+1)]$ such that $a_{k-1,l}$ and $a_{k,l}$ have different signs (call them {\it marked} segments). Then our rectangle $R$ is divided by the marked segments into several connected parts corresponding to the terms of $\{a_{kl}\}$ of the same sign. The interior part of the union of their boundaries has at most $b2^{n-1}$ vertical marked segments and at most $b2^{m-1}$ horizontal ones. Take a positive integer $u$ such that
\begin{align}\label{u}
2^{u}> 8b\tau,
\end{align}
where $\tau:=4\sqrt{T(C+1)^2+1}$. Divide $R$ into $2^{2u}$ equal rectangles of size $2^{m-1-u}\times 2^{n-1+u}$ and consider a half of them in a checkerboard pattern. Suppose that there is no rectangle among them containing at most $2^{n-1-u}/\tau$ vertical marked segments and at most $2^{m-1-u}/\tau$ horizontal ones. Then we must have
\begin{align*}
2^{2u-1}\leq \frac{b2^{m-1}\tau}{2^{m-1-u}}+\frac{b2^{n-1}\tau}{2^{n-1-u}}=2^{u+2}b\tau\leq 4b\tau 2^{u},
\end{align*} 
which contradicts \eqref{u}. So, there is a rectangle $r=[\a_1,\a_2]\times [\beta_1,\beta_2]$ of size $2^{m-1-u}\times 2^{n-1-u}$ with at most $2^{n-1-u}/\tau$ vertical marked segments and at most $2^{m-1-u}/\tau$ horizontal ones inside it. Consider the parts corresponding to the terms of $\{a_{kl}\}$ of the same sign inside $r$. Call the parts whose boundaries intersect the boundary of $r$ by $A$-parts, the other ones, by $B$-parts. Note that there is no marked segment of an $A$-part inside the rectangle $r':=[\frac{3\a_1+\a_2}{4},\frac{\a_1+3\a_2}{4}]\times [\frac{3\beta_1+\beta_2}{4},\frac{\beta_1+3\beta_2}{4}]$. Indeed, otherwise there would exist a broken line of marked segments with either at least $0.25(\a_2-\a_1)=2^{m-3-u}$ horizontal segments or at least $0.25(\beta_2-\beta_1)=2^{n-3-u}$ vertical ones. But this is impossible, since $\tau>4$. The area of all $B$-parts does not exceed $2^{m+n-2-2u}/\tau^2$. Thus, there are at least $2^{m+n-4-2u}(1-4\tau^{-2})$ terms of the same sign in $r'$, so the absolute value of the sum of the terms $\{a_{kl}\}$ in $r'$ is at least
\begin{align*}
2^{n+m-2u-4}\Big(1-\frac{4}{\tau^2}-\frac{4}{\tau^2}T(C+1)^2\Big)\a A_{mn}> 2^{n+m-2u-5}\a A_{mn},
\end{align*}
which concludes the proof of Lemma \ref{le}a) with $c:=2^{-2u-5}\a$ and $v:=u+1$.

A similar argument is valid for $Q_{m+1,n-1}$ in Lemma \ref{le}b), which completes the proof.
\end{proof}

\begin{remark} In the proof of Lemma \ref{le}, for $GM^c_1$ class we only used its one-dimensional $GM$ properties \eqref{cond1'}, and for $GM^c_2$, the corresponding nonsymmetric relations (namely, \eqref{cond1'} with $a_{2k,l}$ in place of $a_{k,l}$).
\end{remark}

\begin{remark}\label{rem} The claim of Lemma \ref{le}a) is no longer true if we substitute the $GM^c_1$ condition \eqref{cond2} for
\begin{align}\label{cond3}
\sum_{m=k}^{2k}\sum_{n=l}^{2l}|\Delta^{11}a_{mn}|\leq C|a_{kl}|.
\end{align}
\end{remark}

\begin{proof}
Indeed, consider the sequence $$a_{mn}:=\frac{(-1)^m}{m}f_m(n),$$
where $f_m(n)$ we define as follows: 
\begin{equation*}
f_m(n)=
\begin{cases}
2^{-m+1},\qquad \log_2 n<\frac{m(m+1)}{2},\\
 2^{-m-t},\qquad\frac{(m+t)^2+m-t}{2}\leq \log_2 n<\frac{(m+t+1)^2+m-t-1}{2},\quad t\in\mathbb{Z}_+.
\end{cases}
\end{equation*}

For such a sequence, condition \eqref{cond1} obviously holds. Consider a rectangle $S_{mn}$ of the form $[m,2m)\times [n,2n)$. The only nonzero $\Delta^{11}a_{kl}$ in this rectangle are  $\Delta^{11}a_{m'-1,n'}$ and $\Delta^{11}a_{m'n'}$, where $n'\in[n,2n):\;\lfloor\log_2 (n')\rfloor=\lfloor \log_2 (n'-1)\rfloor+1$, i.e. $n'$ is a power of two, and 
$$m':=\min\Big\{m\in\mathbb{N}:\;m=\log_2 n'-\frac{k(k+1)}{2},\;k\in\mathbb{Z}_+\Big\}.$$
Note that $|a_{kl}|\leq |a_{mn}|$ for $k\geq m,\;l\geq n$, so $|\Delta^{11}a_{m'n'}|\leq |a_{m'n'}|+|a_{m'+1,n'}|\leq 2|a_{mn}|$, which yields condition \eqref{cond3} with $C=2$.

Assume that the assertion of Lemma \ref{le} holds. Then there must exist a constant $c$ such that for at least $cmn$ squares $[k,k+2)\times [l,l+2)$ in any $S_{mn}$ there holds 
\begin{align}\label{con}
|a_{kl}+a_{k,l+1}+a_{k+1,l}+a_{k+1,l+1}|\geq c|a_{kl}|.
\end{align}
Consider a rectangle $S_{mn}$ with
\begin{align*}
\frac{t(t+1)}{2}+2m\leq \log_2 n\leq \frac{(t+1)(t+2)}{2}-2,
\end{align*}
where $t>4m$ is a positive integer. For any $a_{kl}$ in $S_{mn}$, we have
\begin{align*}
a_{kl}=2^{-t-1}\frac{(-1)^k}{k},
\end{align*}
whence for any $2\times 2$ square $[k,k+2)\times [l,l+2)\subset S_{mn}$
\begin{align*}
|a_{kl}+a_{k,l+1}+a_{k+1,l}+a_{k+1,l+1}|=2^{-t-1}\cdot 2\Big(\frac{1}{k}-\frac{1}{k+1}\Big)=\frac{2}{k+1}|a_{kl}|<\frac{2}{m}|a_{kl}|=o(|a_{kl}|),
\end{align*}
as $m\to \infty$, which leads to a contradiction.
\end{proof}

\begin{lemma}\label{oddeven} For a function $f\in L(-\pi,\pi)$, given the representation
\begin{align*}
f(x,y)=\sum_{i,j=0}^1 f^{ij}(x,y),\quad f^{ij}(-x,y)=(-1)^{i}f^{ij}(x,y),\;f^{ij}(x,-y)=(-1)^{j}f^{ij}(x,y),
\end{align*}
for any $p\in (1,\infty),\;q\in[1,\infty],$ we have
\begin{align*}
\|f\|_{L_{w(p,q)}^q}\asymp \sum_{i,j=0}^{1}\|f^{ij}\|_{L_{w(p,q)}^q}.
\end{align*}
\end{lemma}

\begin{proof}
The $\lesssim$ part is clear, so we have to prove the reverse. 
\\
We start with the case $q<\infty$. Noting that for any pair of functions $g_1,g_2$ there always holds $|g_1|^q+|g_2|^q\lesssim |g_1+g_2|^q+|g_1-g_2|^q$ and recalling that the weight is an even in each variable function, we obtain
\begin{align*}
\|f^{i0}(x,\cdot)\|_{L_{w(p,q)}^q}^q+\|f^{i1}(x,\cdot)\|_{L_{w(p,q)}^q}^q&\lesssim \|(f^{i0}+f^{i1})(x,\cdot)\|_{L_{w(p,q)}^q}^q+\|(f^{i0}-f^{i1})(x,\cdot)\|_{L_{w(p,q)}^q}^q\\
&\asymp \|(f^{i0}+f^{i1})(x,\cdot)\|_{L_{w(p,q)}^q}^q
\end{align*}
for $i=0,1$. Similarly,
\begin{align*}
\sum_{i,j=0}^{1}\|f^{ij}\|_{L_{w(p,q)}^q}^q&\lesssim \|f^{00}+f^{01}+f^{10}+f^{11}\|_{L_{w(p,q)}^q}^q+\|f^{00}+f^{01}-f^{10}-f^{11}\|_{L_{w(p,q)}^q}^q\\
&\asymp \Big\|\sum_{i,j=0}^1 f^{ij}\Big\|_{L_{w(p,q)}^q}^q=\|f\|_{L_{w(p,q)}^q}^q.
\end{align*}
For $q=\infty$, the claim follows from the equalities
\begin{align*}
4f^{ij}(x,y)\equiv f(x,y)+(-1)^if(-x,y)+(-1)^jf(x,-y)+(-1)^{i+j}f(-x,-y).
\end{align*}
\end{proof}

Next we prove a two-dimensional analogue of \cite[L. 2.2]{DMT} (see also the one-dimensional result \cite[Th. 2.4]{S} for Lorentz spaces). Note that similar multidimensional results for Lorentz spaces were obtained in \cite{N-1} and \cite{N-2}.

\begin{lemma}\label{sag2}
Let $\{a_{mn}^{ij}\}_{m,n=1}^{\infty},\;i,j=0,1,$ be the sequence of Fourier coefficients of $f\in L(-\pi,\pi)$. 
Then for any $p\in (1,\infty),\;q\in[1,\infty],$ there holds 
\begin{align*}
\sum_{i,j=0}^1\sum_{m,n=1}^{\infty}\Big(\sup\limits_{k\ge m,\;l\geq n}\frac{1}{kl}\Big|\sum\limits_{s=1}^{k}\sum\limits_{t=1}^{l}a_{st}\Big|\Big)^q(mn)^{\frac{q}{p'}-1} \lesssim \|f\|_{L_{w(p,q)}^q}^q.
\end{align*}
\end{lemma}

\begin{proof}[Proof of Lemma \ref{sag2}]
Note that if we prove the statement of the lemma for odd in each variable functions $f\in L(-\pi,\pi)$, then it will be true for any integrable $f$. Indeed, the relation for such functions implies the same for all functions that are either odd or even in each variable due to the boundedness of the Hilbert transform in weighted Lebesgue spaces. The general case follows then by Lemma \ref{oddeven}. Thus, we can assume that $a_{mn}^{ij}=0$ if $(i,j)\neq (0,0)$ and omit the upper indices of $a_{mn}^{00}$.

According to \cite[(2.4), (2.7)]{DMT}, for any $1< p<\infty,\;1\leq q\leq \infty$, and $m\in\mathbb{N}$, there holds
\begin{align*}
\|I_m(x)\|_{l_{p,q}}:=\bigg\|\frac{\cos \frac{x}{2}(1 - \cos mx )}{m\sin\frac{x}{2}} + \frac{\sin mx }{m}\bigg\|_{l_{p,q}}\lesssim m^{-\frac{1}{p}}.
\end{align*}
Therefore, for any $1< p_1,p_2<\infty,\;1< q\leq \infty$, and $m,n\in\mathbb{N}$, by H\"older's inequality
\begin{align}\label{lebes}
\frac{1}{mn}\Big|\sum\limits_{k=1}^{m}\sum\limits_{l=1}^{n}a_{kl}\Big|  &\leq
\int_0^{\pi}\int_0^{\pi} |f(x,y)I_m(x)I_n(y)| dx dy \nonumber\\
&\leq\int_0^{\pi}|I_n(y)|\Big(\int_0^{\pi} x^{\frac{q}{p_1}-1}|f(x,y)|^qdx\Big)^{\frac{1}{q}}\Big(\int_0^{\pi}x^{\frac{q'}{p_1'}}|I_m(x)|^{q'} dx\Big)^{\frac{1}{q'}}\;dy\nonumber\\
&\lesssim m^{-\frac{1}{p'_1}}\int_0^{\pi}|I_n(y)|\Big(\int_0^{\pi} x^{\frac{q}{p_1}-1}|f(x,y)|^qdx\Big)^{\frac{1}{q}}\;dy\nonumber\\
&\leq m^{-\frac{1}{p'_1}}\Big(\int_0^{\pi}\int_0^{\pi}x^{\frac{q}{p_1}-1}y^{\frac{q}{p_2}-1}|f(x,y)|^q dx dy\Big)^{\frac{1}{q}}\Big(\int_0^{\pi}y^{\frac{q'}{p'_2}-1}|I_n(y)|dy\Big)^{\frac{1}{q'}}\nonumber\\
&\lesssim m^{-\frac{1}{p'_1}}n^{-\frac{1}{p'_2}}\Big(\int_0^{\pi}\int_0^{\pi}x^{\frac{q}{p_1}-1}y^{\frac{q}{p_2}-1}|f(x,y)|^q dx dy\Big)^{\frac{1}{q}}\nonumber\\
&=:m^{-\frac{1}{p'_1}}n^{-\frac{1}{p'_2}}\|f\|_{L_{w((p_1, p_2),q)}^q}.
\end{align}
Similarly, if $q=1$,
\begin{align*}
\frac{1}{mn}\Big|\sum\limits_{k=1}^{m}\sum\limits_{l=1}^{n}a_{kl}\Big|  &\leq 
\int_0^{\pi}\int_0^{\pi} |f(x,y)I_m(x)I_n(y)| dx dy \\
&\leq \sup_{x\in [0,\pi]}x^{\frac{1}{p'_1}}|I_m(x)|\cdot \sup_{y\in [0,\pi]}y^{\frac{1}{p'_2}}|I_n(y)|\cdot\int_0^{\pi}\int_0^{\pi}x^{\frac{1}{p_1}-1}y^{\frac{1}{p_2}-1}|f(x,y)|\;dxdy\\
&\lesssim m^{-\frac{1}{p'_1}}n^{-\frac{1}{p'_2}}\|f\|_{L_{w((p_1,p_2),1)}^1}.
\end{align*}
Thus, for any $1< p_1,p_2<\infty,\;1\leq q\leq \infty$, and $m\in\mathbb{N}$, we obtain
\begin{align}\label{inter1}
m^{\frac{1}{p'_1}}\sup_{n\in\mathbb{N}}n^{\frac{1}{p'_2}}\sup\limits_{k\ge m,\;l\geq n}\frac{1}{kl}\Big|\sum\limits_{s=1}^{k}\sum\limits_{t=1}^{l}a_{st}\Big|\leq C\|f\|_{L_{w((p_1,p_2),q)}^q},
\end{align}
where the constant $C$ does not depend on $m$. 

Now, in order to prove the desired inequality, we will invoke interpolation theory. Recall that the norm of a sequence ${\bf c}:=\{c_k\}_{k=1}^{\infty}$ in the discrete Lorentz space $l_{p,q}$, for $p\in(1,\infty)$ and $q\in (0,\infty]$, is defined as follows
\begin{align*}
\|{\bf c}\|_{l_{p,q}}:=
\begin{cases}
\Big(\sum_{k=1}^{\infty}k^{\frac{q}{p}-1}|c^*_k|^q\Big)^{\frac{1}{q}},\quad\text{if}\;q < \infty, \\
\sup\limits_{k\geq 1}k^{\frac{1}{p}} |c^*_k|,\quad \text{if} \; q = \infty,
\end{cases}
\end{align*}
where $\{c^*_k\}$ stands for the decreasing rearrangement of ${\bf c}$. It follows from \cite[Th. 5.3.1]{BL} that for $\theta\in (0,1)$ and $q\in (0,\infty]$, for the discrete Lorentz spaces $l_{p_1,\infty}$ and $l_{p_2,\infty},\;0<p_1<p_2\leq \infty,$ with $\theta/p_1+(1-\theta)/p_2=1/p$, we have
\begin{align}\label{bl1}
(l_{{p_{1}},\infty},l_{{p_{2}},\infty})_{\theta,q}= l_{p,q}.
\end{align}
For the Lebesgue spaces $L_{w((p_{11},p_{21}),q)}^q$ and $L_{w((p_{21},p_{22}),q)}^q,\;q\in(0,\infty],$ (see \eqref{lebes}), with $\theta/p_{11}+(1-\theta)/p_{12}=1/p_1,\;\theta/p_{21}+(1-\theta)/p_{22}=1/p_2$, \cite[Th. 5.4.1]{BL} gives 
\begin{align}\label{bl2}
(L_{w((p_{11},p_{21}),q)}^q, L_{w((p_{12},p_{22}),q)}^q)_{\theta,q}=L_{w((p_1,p_2),q)}^q.
\end{align}
For any fixed $m_0\in\mathbb{N}$, in light of the monotonicity of $\sup\limits_{k\ge m_0,\;l\geq n}\frac{1}{kl}\Big|\sum\limits_{s=1}^{k}\sum\limits_{t=1}^{l}a_{st}\Big|$ in $n$, \eqref{inter1} is equivalent to
\begin{align}\label{inter2}
m_0^{\frac{1}{p'_1}}\Big\|\Big\{\sup\limits_{k\ge m_0,\;l\geq n}\frac{1}{kl}\Big|\sum\limits_{s=1}^{k}\sum\limits_{t=1}^{l}a_{st}\Big|\Big\}_{n=1}^{\infty}\Big\|_{l_{p'_2,\infty}}\leq C\|f\|_{L_{w((p_1,p_2),q)}^q}.
\end{align}
Fix now $p_1,p_2\in (1,\infty)$ and $q\in [1,\infty]$. Take $\theta\in (0,1)$ and $p_{11}<p_{12},\;p_{21}<p_{22}$ such that $\theta/p_{11}+(1-\theta)/p_{12}=1/p_1$ and $\theta/p_{21}+(1-\theta)/p_{22}=1/p_2$. Note that, for any fixed $m_0$, the operator $$T_{m_0}f=\Big\{\sup\limits_{k\ge m_0,\;l\geq n}\frac{1}{kl}\Big|\sum\limits_{s=1}^{k}\sum\limits_{t=1}^{l}a_{st}\Big|\Big\}_{n=1}^{\infty}$$ is sublinear and that due to \eqref{inter2} $$T_{m_0}:L_{w((p_1,p_{21}),q)}^q\to l_{p'_{21},\infty}\quad\text{and}\quad T_{m_0}:L_{w((p_1,p_{22}),q)}^q\to l_{p'_{22},\infty},$$ 
where the involved constants do not depend on $m_0$. Then it follows from \cite[Th. 6]{Ma}, \eqref{bl1}, and \eqref{bl2} that
\begin{align*}
T_{m_0}:L_{w((p_1,p_{2}),q)}^q&=(L_{w((p_1,p_{21}),q)}^q,L_{w((p_1,p_{22}),q)}^q)_{\theta,q}\to (l_{p'_{21},\infty},l_{p'_{22},\infty})_{\theta,q}=l_{p'_2,q},
\end{align*}
so we arrive at
\begin{align}\label{inter3}
m^{\frac{1}{p'_1}}\Big\|\Big\{\sup\limits_{k\ge m,\;l\geq n}\frac{1}{kl}\Big|\sum\limits_{s=1}^{k}\sum\limits_{t=1}^{l}a_{st}\Big|\Big\}_{n=1}^{\infty}\Big\|_{l_{p_2,q}}\lesssim\|f\|_{L_{w((p_1,p_2),q)}^q},
\end{align}
for any $m$. Now we note that $$\Big\|\Big\{\sup\limits_{k\ge m,\;l\geq n}\frac{1}{kl}\Big|\sum\limits_{s=1}^{k}\sum\limits_{t=1}^{l}a_{st}\Big|\Big\}_{n=1}^{\infty}\Big\|_{l_{p_2,q}}=\Big(\sum_{n=1}^{\infty}n^{\frac{q}{p_2}-1}\Big(\sup\limits_{k\ge m,\;l\geq n}\frac{1}{kl}\Big|\sum\limits_{s=1}^{k}\sum\limits_{t=1}^{l}a_{st}\Big|\Big)^q\Big)^{1/q}$$ is decreasing in $m$ for any $p_2\in (1,\infty)$ and that the operator $$Tf=\Big\{\Big(\sum_{n=1}^{\infty}n^{\frac{q}{p_2}-1}\Big(\sup\limits_{k\ge m,\;l\geq n}\frac{1}{kl}\Big|\sum\limits_{s=1}^{k}\sum\limits_{t=1}^{l}a_{st}\Big|\Big)^q\Big)^{1/q}\Big\}_{m=1}^{\infty}$$ is sublinear. Since according to \eqref{inter3} we have $$T:L_{w((p_{11},p_{2}),q)}^q\to l_{p'_{11},\infty}\quad\text{and}\quad T:L_{w((p_{12},p_{2}),q)}^q\to l_{p'_{12},q},$$ 
we can once again apply \cite[Th. 6]{Ma} and obtain
\begin{align*}
T:L_{w((p_1,p_{2}),q)}^q=(L_{w((p_{11},p_{2}),q)}^q,L_{w((p_{12},p_{2}),q)}^q)_{\theta,q}\to (l_{p'_{11},\infty},l_{p'_{12},\infty})_{\theta,q}=l_{p'_1,q}.
\end{align*}
The latter means that
\begin{align*}
\Big\|\Big\{\Big(\sum_{n=1}^{\infty}n^{\frac{q}{p_2}-1}\Big(\sup\limits_{k\ge m,\;l\geq n}\frac{1}{kl}\Big|\sum\limits_{s=1}^{k}\sum\limits_{t=1}^{l}a_{st}\Big|\Big)^q\Big)^{\frac{1}{q}}\Big\}_{m=1}^{\infty}\Big\|_{l_{p_1,q}}\lesssim\|f\|_{L_{w((p_1,p_2),q)}^q},
\end{align*}
whence the claim follows by putting $p_1=p_2=p$.
\end{proof}

\begin{proof} [Proof of Theorem \ref{m}] In light of Lemma \ref{oddeven} it suffices to prove the theorem only for either odd or even in each variable functions, omitting therefore the upper indices of $a_{mn}$.

We start with the part a). Due to Lemma \ref{sag2} there holds
\begin{align}\label{mean}
 \|f\|_{L_{w(p,q)}^q}^q &\gtrsim \sum_{m,n=1}^{\infty}\Big(\sup\limits_{k\ge m,\;l\geq n}\frac{1}{kl}\Big|\sum\limits_{s=1}^{k}\sum\limits_{t=1}^{l}a_{st}\Big|\Big)^q(mn)^{\frac{q}{p'}-1}\nonumber\\
 &\asymp \sum_{m,n=0}^{\infty} 2^{(m+n)\frac{q}{p'}}\Big(\sup_{k\ge 2^m,\;l\geq 2^n}
\frac{1}{kl}\Big|\sum_{i=1}^k\sum_{j=1}^l a_{ij}\Big|\Big)^q=: \sum_{m,n=0}^{\infty} P_{mn}.
\end{align}
Denote
$$
W_{mn} := \sum_{k=2^m}^{2^{m+1}-1}\sum_{l=2^n}^{2^{n+1}-1}|a_{kl}|^q(kl)^{\frac{q}{p'}-1}.
$$

Consider first $GM^c_1$ sequences. Let us fix some $T>1$. We call a pair $(m,n)$ {\it good} (we write $(m,n)\in G$), if either $mn=0$ or $A_{m-1,n-1}\leq TA_{mn}$.
We have
\begin{align*}
\sum_{k,l=1}^{\infty}|a_{kl}|^q(kl)^{\frac{q}{p'}-1} &= \sum_{m,n=0}^{\infty} W_{mn}
 \le \sum_{m=0}^{\infty} W_{m0} + \sum_{n=0}^{\infty} W_{0n}+
\sum_{(m,n) \in G\cap \mathbb{N}^2} W_{mn} + \sum_{(m,n) \in G } \sum_{(k,l) \in B_{mn}} W_{kl}\\
&=: J_1 + J_2 + J_3 + J_4,
\end{align*}
where $B_{mn},\;(m,n)\in G,$ stands for the set of all pairs $(k,l)\notin G$ such that $k=m+t,\;l=n+t$ for some $t\in\mathbb{N}$.

According to the one-dimentional Hardy-Littlewood theorem for $GM$ sequences \cite[Th. 1.2]{DMT}, we obtain
\begin{align}\label{J1}
J_1=\sum_{m=0}^{\infty} W_{m0}=\sum_{k=1}^{\infty}|a_{k1}|^q k^{\frac{q}{p'}-1}\lesssim \|g\|_{L_{w(p,q)}^q}^q\lesssim \|f\|_{L_{w(p,q)}^q}^q,
\end{align}
where $g(x)=\int_{-\pi}^{\pi}f(x,y)\sin y\;dy$. A similar estimate is valid for $J_2$.

Consider a pair $(m,n)\in G\cap \mathbb{N}^2$. Denote the rectangles we constructed in Lemma \ref{le}a) $[s_{mn}^1,s_{mn}^2]\times [t_{mn}^1,t_{mn}^2]$, so we have
\begin{align*}
P_{m-1,n-1}&=2^{(m+n-2)\frac{q}{p'}}\Big(\sup_{k\ge 2^{m-1},\;l\geq 2^{n-1}}
\frac{1}{kl}\Big|\sum_{i=1}^k\sum_{j=1}^l a_{ij}\Big|\Big)^q\\
&\gtrsim 2^{(m+n)\frac{q}{p'}-(m+n)q}\Bigg(
\Big|\sum_{i=1}^{s_{mn}^1-1}\sum_{j=1}^{t_{mn}^1-1} a_{ij}\Big|^q+\Big|\sum_{i=1}^{s_{mn}^1-1}\sum_{j=1}^{t_{mn}^2} a_{ij}\Big|^q+\Big|\sum_{i=1}^{s_{mn}^2}\sum_{j=1}^{t_{mn}^1-1} a_{ij}\Big|^q+\Big|\sum_{i=1}^{s_{mn}^2}\sum_{j=1}^{t_{mn}^2} a_{ij}\Big|^q\Bigg)\\
&\gtrsim 2^{(m+n)\frac{q}{p'}-(m+n)q}
\Big|\sum_{i=s_{mn}^1}^{s_{mn}^2}\sum_{j=t_{mn}^1}^{t_{mn}^2} a_{ij}\Big|^q\gtrsim 2^{(m+n)\frac{q}{p'}}A_{mn}^q\gtrsim W_{mn}.
\end{align*}
Here we used the inequality
\begin{align*}
|x+y+z+t|+|x+y|+|x+z|+|x|\ge |z+t|+|z|\ge |t|,
\end{align*}
valid for any $x,y,z,t\in\mathbb{C}$.

Hence, using \eqref{mean}, we obtain
\begin{align}\label{J3}
J_3 =\sum_{(m,n)\in G\cap \mathbb{N}^2} W_{mn} \lesssim \sum_{(m,n)\in G\cap \mathbb{N}^2} P_{m-1,n-1} \leq\|f\|_{L_{w(p,q)}^q}^q.
\end{align}
Finally, combining \eqref{J1}, the similar estimate for $J_2$, and \eqref{J3}, we derive
\begin{align*}
J_4\leq \sum_{(m,n)\in G}W_{mn}\sum_{j=1}^{\infty}T^{-j}\leq \frac{1}{1-T^{-1}}(J_1+J_2+J_3)\lesssim \|f\|_{L_{w(p,q)}^q}^q,
\end{align*}
which concludes the proof of the first part for the case of $GM^c_1$. 

If we replace $GM^c_1$ by $GM^c_2$, i.e. \eqref{cond2} by \eqref{cond22}, we change the definition of a good pair of numbers to the following one: we call a pair $(m,n)$ good, if either $mn=0$ or $A_{m+1,n-1}\leq TA_{mn}$. The rest of the proof is the same in light of Lemma \ref{le}b) with the only changes: now $B_{mn},\;(m,n)\in G,$ stands for the set of all pairs $(k,l)\notin G$ such that $k=m-t,\;l=n+t$ for some $t\in\mathbb{N}$ and $P_{m-1,n-1}$ in \eqref{J3} becomes $P_{m+1,n-1}$.

Turn now to the part b). Note that if $\{a_{mn}\}\in GM^c_1\cup GM^c_2$ and $\sum_{m,n=1}^{\infty}|a_{mn}|^q(mn)^{\frac{q}{p'}-1}<\infty$, then we have $\sum_{k=1}^{\infty}\sum_{l=1}^{\infty}|\Delta^{11}a_{kl}|<\infty$, which implies that the corresponding trigonometric series converges in the Pringsheim sense everywhere on $(0,2\pi)^2$ and is the Fourier series of its sum (see \cite[L. 4]{D-1}). Indeed, under condition \eqref{cond2} we have by \eqref{ineq} and H\"older's inequality
\begin{align*}
&\sum_{k,l=1}^{\infty}|\Delta^{11}a_{kl}|\lesssim\sum_{k=0}^{\infty}|a_{2^k,2^k}|\lesssim\sum_{k=0}^{\infty}|a_{2^k,2^k}|\sum_{m=2^{k-1}}^{2^k}\sum_{n=2^{k-1}}^{2^k}(mn)^{-1}\lesssim\sum_{m,n=1}^{\infty}|a_{mn}|(mn)^{-1}\\
=&\sum_{m,n=1}^{\infty}|a_{mn}|(mn)^{\frac{1}{p'}-\frac{1}{q}}(mn)^{-\frac{1}{p'}-\frac{1}{q'}}\lesssim\Big(\sum_{m,n=1}^{\infty}|a_{mn}|^q(mn)^{\frac{q}{p'}-1}\Big)^{\frac{1}{q}}\Big(\sum_{m,n=1}^{\infty}(mn)^{-\frac{q'}{p'}-1}\Big)^{\frac{1}{q'}}<\infty,
\end{align*}
and similarly under \eqref{cond22},
\begin{align*}
\sum_{k=1}^{\infty}\sum_{l=1}^{\infty}|\Delta^{11}a_{kl}|&\lesssim\sum_{k=0}^{\infty}|a_{2^{k+1},2^k}|\lesssim\sum_{k=0}^{\infty}|a_{2^{k+1},2^k}|\sum_{m=2^{k+1}}^{2^{k+2}}\sum_{n=2^{k-1}}^{2^k}(mn)^{-1}\lesssim \sum_{m,n=1}^{\infty}|a_{mn}|(mn)^{-1}<\infty.
\end{align*}
 
We will provide the proof only for the system $\{\sin mx,\sin ny\}$, the other cases will follow then from boundedness of Hilbert transform in weighted Lebesgue spaces.

For $(x,y)\in (\frac{\pi}{m+1},\frac{\pi}{m}]\times (\frac{\pi}{n+1},\frac{\pi}{n}]$, we have

\begin{align*}
|f(x,y)|&=\Big|\sum_{k=1}^{\infty}\sum_{l=1}^{\infty}a_{kl}\sin kx\sin ly\Big|\leq xy\sum_{k=1}^m\sum_{l=1}^n kl|a_{kl}|+x\sum_{k=1}^m k\sum_{l=n}^{\infty}|a_{kl}-a_{k,l+1}||\tilde{D}_l(y)-\tilde{D}_n(y)|\nonumber\\
&+y\sum_{l=1}^n l\sum_{k=m}^{\infty}|a_{kl}-a_{k+1,l}||\tilde{D}_k(x)-\tilde{D}_m(x)|\nonumber\\
&+\sum_{k=m}^{\infty}\sum_{l=n}^{\infty}|\Delta^{11}a_{kl}|\cdot|(\tilde{D}_k(x)-\tilde{D}_m(x))(\tilde{D}_l(y)-\tilde{D}_n(y))|\lesssim \frac{1}{mn}\sum_{k=1}^m\sum_{l=1}^n kl|a_{kl}|\nonumber\\
&+\frac{n}{m}\sum_{k=1}^m k\sum_{l=n}^{\infty}|a_{kl}-a_{k,l+1}|+\frac{m}{n}\sum_{l=1}^n l\sum_{k=m}^{\infty}|a_{kl}-a_{k+1,l}|+mn\sum_{k=m}^{\infty}\sum_{l=n}^{\infty}|\Delta^{11}a_{kl}|.
\end{align*}
Applying condition \eqref{cond2}, we derive 
\begin{align*}
|f(x,y)|&\lesssim \frac{1}{mn}\sum_{k=1}^m\sum_{l=1}^n kl|a_{kl}|+\frac{n}{m}\sum_{k=1}^m k\sum_{t=0}^{\infty}|a_{k,2^tn}|+\frac{m}{n}\sum_{l=1}^n l\sum_{t=0}^{\infty}|a_{2^tm,l}|+mn\sum_{t=0}^{\infty}|a_{2^tm,2^tn}|\nonumber\\
&\lesssim \frac{1}{mn}\sum_{k=1}^m\sum_{l=1}^n kl|a_{kl}|+\frac{n}{m}\sum_{k=1}^m k\sum_{l=\lceil n/2\rceil}^{\infty}\frac{|a_{kl}|}{l}+\frac{m}{n}\sum_{l=1}^n l\sum_{k=\lceil m/2\rceil}^{\infty}\frac{|a_{kl}|}{k}+mn\sum_{k=\lceil m/2\rceil}^{\infty}\sum_{l=\lceil n/2\rceil}^{\infty}\frac{|a_{kl}|}{kl}.
\end{align*}
In turn, \eqref{cond22} yields
\begin{align*}
|f(x,y)|&\lesssim \frac{1}{mn}\sum_{k=1}^m\sum_{l=1}^n kl|a_{kl}|+\frac{n}{m}\sum_{k=1}^m k\sum_{t=0}^{\infty}|a_{k,2^tn}|+\frac{m}{n}\sum_{l=1}^n l\sum_{t=0}^{\infty}|a_{2^{t+1}m,l}|+mn\sum_{t=0}^{\infty}|a_{2^{t+1}m,2^tn}|\nonumber\\
&\lesssim \frac{1}{mn}\sum_{k=1}^m\sum_{l=1}^n kl|a_{kl}|+\frac{n}{m}\sum_{k=1}^m k\sum_{l=\lceil n/2\rceil}^{\infty}\frac{|a_{kl}|}{l}+\frac{m}{n}\sum_{l=1}^n l\sum_{k=2m}^{\infty}\frac{|a_{kl}|}{k}+mn\sum_{k=2m}^{\infty}\sum_{l=\lceil n/2\rceil}^{\infty}\frac{|a_{kl}|}{kl}.
\end{align*}
Hence, in both cases we get
\begin{align}\label{bo}
|f(x,y)|&\lesssim \frac{1}{mn}\sum_{k=1}^m\sum_{l=1}^n kl|a_{kl}|+\frac{n}{m}\sum_{k=1}^m k\sum_{l=\lceil n/2\rceil}^{\infty}\frac{|a_{kl}|}{l}+\frac{m}{n}\sum_{l=1}^n l\sum_{k=\lceil m/2\rceil}^{\infty}\frac{|a_{kl}|}{k}+mn\sum_{k=\lceil m/2\rceil}^{\infty}\sum_{l=\lceil n/2\rceil}^{\infty}\frac{|a_{kl}|}{kl}\nonumber\\
&=:I^1_{m,n}+I^2_{m,n}+I^3_{m,n}+I^4_{m,n}.
\end{align}
Thus, for $q<\infty$, denoting $\a:=1-q/p$, we obtain
\begin{align*}
\|f\|_{L_{p,q}^q}^q &\asymp \int\limits_0^{\pi}\int\limits_0^{\pi}(xy)^{-\a}|f(x,y)|^q\;dx dy\lesssim  \sum_{m=1}^{\infty}\sum_{n=1}^{\infty}\int\limits_{\frac{\pi}{m+1}}^{\frac{\pi}{m}}\int\limits_{\frac{\pi}{n+1}}^{\frac{\pi}{n}}(xy)^{-\a}(I^1_{m,n}+I^2_{m,n}+I^3_{m,n}+I^4_{m,n})^q\;dx dy\\
&\asymp \sum_{m=1}^{\infty}\sum_{n=1}^{\infty}(mn)^{\a-2}((I^1_{m,n})^q+(I^2_{m,n})^q+(I^3_{m,n})^q+(I^4_{m,n})^q).
\end{align*}
Recall the Hardy-type inequalities for power weights (see, for instance, \cite[(0.6), (0.10), (1.102)]{KP}) %were obtained for $a_n\geq 0,\;\lambda_n>0,$ and
%\cite[(1.9), (1.10)]{L}
 for $q\geq 1$:
\begin{align}\label{Hard1}
\sum_{n=1}^{\infty}n^{\gamma}\Big(\sum_{k=1}^na_k\Big)^q\lesssim_q\sum_{n=1}^{\infty}n^{\gamma+q}a_n^q,\quad\text{for}\; \gamma<-1,
\end{align}
and its dual,
\begin{align}\label{Hard2}
\sum_{n=1}^{\infty}n^{\gamma}\Big(\sum_{k=n}^{\infty}a_k\Big)^q\lesssim_q\sum_{n=1}^{\infty}n^{\gamma+q}a_n^q,\quad\text{for}\;\gamma>-1.
\end{align}
%Using these inequalities
Using \eqref{Hard1} in each variable we arrive at
\begin{align*}
&\sum_{m=1}^{\infty}\sum_{n=1}^{\infty}(mn)^{\a-2}(I^1_{m,n})^q=\sum_{m=1}^{\infty}m^{\a-2-q}\sum_{n=1}^{\infty}n^{\a-2-q}\Big(\sum_{l=1}^nl\sum_{k=1}^m k|a_{kl}|\Big)^q\\
%&\lesssim \sum_{m=1}^{\infty}m^{\a-2-q}\sum_{n=1}^{\infty}n^{(\a-2-q)(1-q)}n^q\Big(\sum_{k=1}^m k|a_{kn}|\Big)^q\Big(\sum_{t=n}^{\infty}t^{\a-2-q}\Big)^q\\
&\lesssim \sum_{n=1}^{\infty}n^{\a-2+q}\sum_{m=1}^{\infty}m^{\a-2-q}\Big(\sum_{k=1}^m k|a_{kn}|\Big)^q\lesssim \sum_{m=1}^{\infty}\sum_{n=1}^{\infty}(mn)^{\a-2+q}|a_{mn}|^q\\
%&\lesssim \sum_{n=1}^{\infty}n^{\a-2+q}\sum_{m=1}^{\infty}m^{(\a-2-q)(1-q)}(m|a_{mn}|)^q\Big(\sum_{t=m}^{\infty} t^{\a-2-q}\Big)^q\asymp 
\end{align*}
and
\begin{align*}
\sum_{m=1}^{\infty}\sum_{n=1}^{\infty}(mn)^{\a-2}(I^2_{m,n})^q&=\sum_{m=1}^{\infty}m^{\a-2-q}\sum_{n=1}^{\infty}n^{\a-2+q}\Big(\sum_{l=\lceil n/2\rceil}^{\infty}\frac{1}{l}\sum_{k=1}^m k|a_{kl}|\Big)^q\\
&\asymp \sum_{m=1}^{\infty}m^{\a-2-q}\sum_{n=1}^{\infty}n^{\a-2+q}\Big(\sum_{l=n}^{\infty}\frac{1}{l}\sum_{k=1}^m k|a_{kl}|\Big)^q\\
%&\lesssim\sum_{m=1}^{\infty}m^{\a-2-q}\sum_{n=1}^{\infty}n^{(\a-2+q)(1-q)}\frac{1}{n^q}\Big(\sum_{k=1}^m k|a_{kn}|\Big)^q\Big(\sum_{t=1}^{n}t^{\a-2+q}\Big)^q\\
&\lesssim \sum_{n=1}^{\infty}n^{\a-2+q}\sum_{m=1}^{\infty}m^{\a-2-q}\Big(\sum_{k=1}^m k|a_{kn}|\Big)^q\\
%&\lesssim \sum_{n=1}^{\infty}n^{\a-2+q}\sum_{m=1}^{\infty}m^{(\a-2-q)(1-q)}m^q|a_{mn}|^qm^{(\a-1-q)q}\\
&\lesssim\sum_{m=1}^{\infty}\sum_{n=1}^{\infty}(mn)^{\a-2+q}|a_{mn}|^q,
\end{align*}
where we used inequality \eqref{ineq}. The similar estimate holds for $I^3$. And finally, due to \eqref{Hard2}
\begin{align*}
\sum_{m=1}^{\infty}\sum_{n=1}^{\infty}(mn)^{\a-2}(I^4_{m,n})^q&=\sum_{m=1}^{\infty}m^{\a-2+q}\sum_{n=1}^{\infty}n^{\a-2+q}\Big(\sum_{l=\lceil n/2\rceil}^{\infty}\frac{1}{l}\sum_{k=\lceil m/2\rceil}^{\infty} \frac{|a_{kl}|}{k}\Big)^q\\
&\asymp \sum_{m=1}^{\infty}m^{\a-2+q}\sum_{n=1}^{\infty}n^{\a-2+q}\Big(\sum_{l=n}^{\infty}\frac{1}{l}\sum_{k=m}^{\infty} \frac{|a_{kl}|}{k}\Big)^q\\
%&\lesssim \sum_{m=1}^{\infty}m^{\a-2+q}\sum_{n=1}^{\infty}n^{(\a-2+q)(1-q)}\frac{1}{n^q}\Big(\sum_{k=m}^{\infty} \frac{|a_{kn}|}{k}\Big)^q n^{(\a-1+q)q}\\
&\lesssim \sum_{m=1}^{\infty}m^{\a-2+q}\sum_{n=1}^{\infty}n^{\a-2+q}\Big(\sum_{k=m}^{\infty} \frac{|a_{kn}|}{k}\Big)^q \\
&\lesssim \sum_{m=1}^{\infty}\sum_{n=1}^{\infty}(mn)^{\a-2+q}|a_{mn}|^q,
\end{align*}
which completes the proof for the case $q\in[1,\infty)$. For $q=\infty,$ using \eqref{bo} we can write
\begin{align*}
\sup_{(x,y)\in (\frac{\pi}{m+1},\frac{\pi}{m}]\times (\frac{\pi}{n+1},\frac{\pi}{n}]}(xy)^{\frac{1}{p}}|f(x,y)|\leq (mn)^{-\frac{1}{p}}(I^1_{m,n}+I^2_{m,n}+I^3_{m,n}+I^4_{m,n}).
\end{align*}
Next,
\begin{align*}
(mn)^{-\frac{1}{p}}I^1_{m,n}&=(mn)^{-\frac{1}{p}-1}\sum_{k=1}^m\sum_{l=1}^{n} kl|a_{kl}|\\
&\leq (mn)^{-\frac{1}{p}-1}\sum_{k=1}^m\sum_{l=1}^{n} (kl)^{\frac{1}{p}} \sup_{k,l}\Big((kl)^{\frac{1}{p'}}|a_{kl}|\Big)\lesssim \sup_{k,l}\Big((kl)^{\frac{1}{p'}}|a_{kl}|\Big).
\end{align*}
We also have
\begin{align*}
(mn)^{-\frac{1}{p}}&I^2_{m,n}=(mn)^{-\frac{1}{p}}\frac{n}{m}\sum_{k=1}^m\sum_{l=\lceil n/2\rceil}^{\infty}\frac{k}{l}|a_{kl}|\lesssim \sup_{k,l}\Big((kl)^{\frac{1}{p'}}|a_{kl}|\Big),
\end{align*}
and the similar estimate for $I^3$. Finally,
\begin{align*}
(mn)^{-\frac{1}{p}}&I^4_{m,n}=(mn)^{-\frac{1}{p}}mn\sum_{k=\lceil m/2\rceil}^{\infty}\sum_{l=\lceil n/2\rceil}^{\infty}\frac{|a_{kl}|}{kl}\lesssim \sup_{k,l}\Big((kl)^{\frac{1}{p'}}|a_{kl}|\Big),
\end{align*}
which completes the proof of the theorem.
\end{proof}

\begin{remark} For the spaces $L_{w(p,q)}^q(0,2\pi)$ in place of $L_{w(p,q)}^q(-\pi,\pi)$, the assertion of Theorem \ref{m} still holds for $q\leq p$ but fails for $q>p$.

Indeed, for $q>p$ it suffices to consider the one-dimensional sine series
\begin{align*}
f(x):=\sum_{k=1}^{\infty}k^{-\frac{1}{p'}}\log^{-\frac{1}{p}}(k+2)\sin kx=:\sum_{k=1}^{\infty} a_{k}\sin kx.
\end{align*}
We have $\sum |a_k|^p k^{p-2}= \sum k^{-1}\log^{-1} (k+2)=\infty$, so by the Hardy-Littlewood theorem $f\notin L_p$, whence $\|f\|_{L_{w(p,q)}^q(0,2\pi)}\gtrsim \|f\|_{L_p(\pi,2\pi)}=\infty$. On the other hand, 
$$\|f\|_{L_{w(p,q)}^q(-\pi,\pi)}\asymp\sum |a_k|^qk^{\frac{q}{p'}-1}= \sum k^{-1}\log^{-\frac{q}{p}} (k+2)<\infty.$$
%The two-dimensional case follows now by taking the product $f(x)f(y)$.

However, for $q\leq p$, there holds $x^{q/p-1}\gtrsim 1$, so that
\begin{align*}
\|f\|_{L_{w(p,q)}^q(0,2\pi)}\asymp \|f\|_{L_{w(p,q)}^q(0,\pi)}+\|f\|_{L_{w(p,q)}^q(\pi,2\pi)}\asymp \|f\|_{L_{w(p,q)}^q(0,\pi)}\asymp \|f\|_{L_{w(p,q)}^q(-\pi,\pi)}.
\end{align*}
 
The reason of the failure of the Hardy-Littlewood relation here is that the function in case is supposed to be periodic, while a power weight is not. Thus, if one deals with weighted Lebesgue spaces on $[0,2\pi]^2$, it makes more sense to consider a weight of the type $|\sin x|^{\alpha}$ in place of $|x|^{\alpha}$, which was in fact done by many authors. Note that for a power weight, weighted integrability at $2\pi$ is equivalent to integrability at zero without weight, so, as in the example above, one has to additionally check integrability at zero. 
\end{remark}

\section{Sharpness of the result}

\begin{theorem}\label{re} For $p>2,\;q\geq p$, the claim of Theorem \ref{m}a) does not hold if we replace the $GM^c_2$ condition \eqref{cond22} by
\begin{align}\label{cond5}
\sum_{m=k}^{2k}\sum_{n=l}^{2l}|\Delta^{11}a_{mn}|\leq C|a_{2k,l}|.
\end{align}
\end{theorem}

\begin{proof} Assume that $p>2$ and consider the sequence
$$a_{mn}:=\frac{(-1)^{\delta_m}}{m^{\gamma}}g_m(n),$$
where $\gamma>0$, $\delta_m\in\{0,1\}$ are to be chosen later, and $g_m(n)=g_m(n,p')$ we define as follows 
\begin{equation*}
g_m(n):=
\begin{cases}
(-1)^{\delta_m}m^{-3}n^{-\frac{1}{p'}},\;\; \log_2 n<m(m+1)p',\\
 2^{-(m+t)^2-3(m+t)},\;\;((m+t)^2+m-t)p'\leq \log_2 n<((m+t)^2+3m+t)p',\;\; t\in\mathbb{Z}_+.
\end{cases}
\end{equation*}
In other words, the functions $g_m$ are constructed in the following way. First, we divide $[1,\infty)$ into intervals $I_j,\;j=0,1,...,$ so that $I_j:=\{x: 2p'j\leq \log_2 x<2p'(j+1)\}$. After that consider the lower-triangular infinite down and to the right matrix that is filled by all positive integers in increasing order going down and to the right.
\begin{align*}
\begin{matrix}
1 & & &  \\
2 & 3 &&\\
4&5&6&\\
7&8&9&10\\
\vdots &\vdots &\vdots &\vdots
\end{matrix}
\end{align*}
Next, for any $j$ we asign it the integer $i=i(j)$ if it is $i$th column that contains the element $j$. Fix some $m$ and consider the values $g_m(1),g_m(2),...$. While $i(j)\neq m$, we have $g_m(n)=(-1)^{\delta_m}m^{-3}n^{-1/p'}$ for $n\in I_j$. Once $i(j)$ becomes equal to $m$ for the first time, that is, when $\log_2 n\geq m(m+1)p'$ for the first time, we get $g_m(n)=2^{-m^2-3m}$ and this value does not change till $i(j)$ becomes equal to $m$ again and $n\in I_j$. When $i(j)$ becomes equal to $m$ for the $(s+1)$th time, the value $g_m(n)$ changes for $2^{-(m+s)^2-3(m+s)}$ (see Figure \ref{pic} for a scheme of changes of absolute values of $g_m(n)$). 

\begin{figure}[H]
\center{\includegraphics[width=0.78\linewidth]{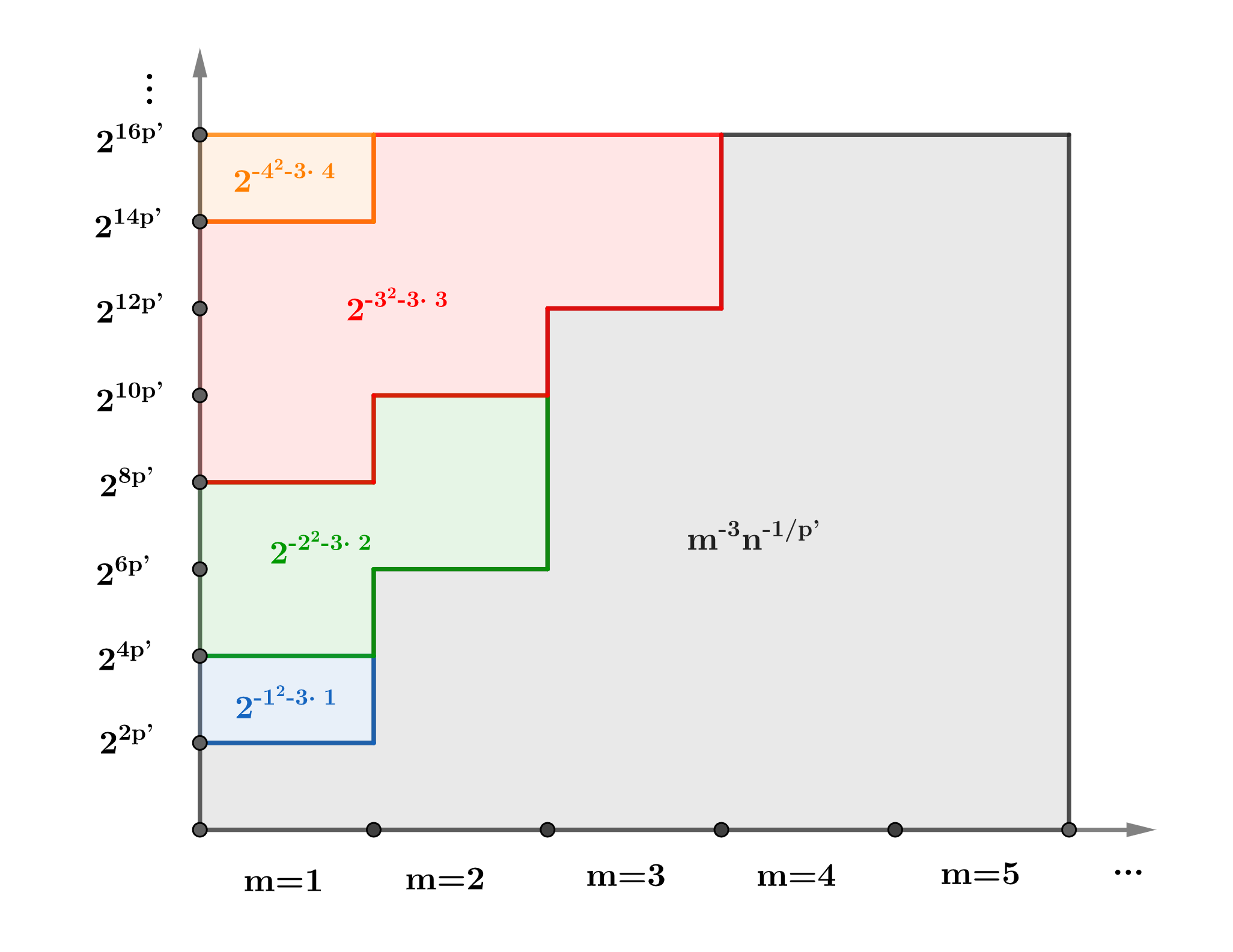}}
\caption{}
\label{pic}
\end{figure}

Fix $n\in I_j$ for some $j$ and consider $g_1(n),g_2(n),...$ Let $k$ be such that $g_m(n)$ has type $1$ if $1\leq m\leq k$ and type $0$ if $m\geq k+1$. Then
\begin{align}\label{a_g}
|g_{m}(n)|\lesssim |g_{m'}(n)|,\quad\text{for}\; k+1\leq m<m'\leq 2m.
\end{align}
Denote $m_0:=i(j+1)$. If $m_0=k+1$, then $g_1(n)=g_2(n)=...=g_{k}(n)=2^{-(k+1)^2-3(k+1)}$, otherwise, $g_m(n)=2^{-(k+1)^2-3(k+1)}$ for $m\leq m_0-1$ and $g_m(n)=2^{-k^2-3k}$ for $m_0\leq m\leq k$. Let us compare $g_k(n)$ and $g_{k+1}(n)$. There are two cases.

Case 1. $m_0=i(j+1)=k+1$. Then
\begin{align*}
|g_{k+1}(n)|=(k+1)^{-3}n^{-\frac{1}{p'}}\gtrsim (k+1)^{-3}2^{-(k+1)(k+2)}\gtrsim 2^{-(k+1)^2-3(k+1)}=g_k(n).
\end{align*}

Case 2. $m_0=i(j+1)<k+1$. Then
\begin{align*}
|g_{k+1}(n)|=(k+1)^{-3}n^{-\frac{1}{p'}}\gtrsim (k+1)^{-3}2^{-k(k+1)-m_0}\gtrsim 2^{-k^2-3k}=g_k(n).
\end{align*}
Thus, in both cases we obtain $0<g_1(n)\leq g_2(n)\leq ... \leq g_k(n)\lesssim |g_{k+1}(n)|$, whence in light of \eqref{a_g},
\begin{align}\label{mon}
|g_m(n)|\lesssim |g_{m'}(n)|,\quad\text{for all}\;m<m'\leq 2m.
\end{align}
It remains to note that for a fixed $m$, we have for $n_m:=\lceil 2^{m(m+1)p'}\rceil-1$ that $$|g_m(n_m)|=m^{-3}n_m^{-\frac{1}{p'}}\asymp m^{-3}2^{-m(m+1)}\gtrsim 2^{-m^3-3m}=g_m(n_m+1)$$
and for other $n$ there holds $g_m(n)\geq g_m(n+1)$. So, over all $|a_{mn}|$ in $r_{kl}:=[k,2k]\times [l,2l]$, the maximal is up to a constant $|a_{2k,l}|$. 

Further we note that the constructed sequence clearly satisfies \eqref{cond1}.

To prove that our sequence belongs to $GM^c_2$, let us estimate $\sum_{m=k}^{2k}\sum_{n=l}^{2l}|\Delta^{11}a_{mn}|$. Consider a quadruple 
\begin{align*}
\begin{matrix}
a_{m,n+1} & a_{m+1,n+1}\\
a_{mn} & a_{m+1,n}
\end{matrix}
\end{align*}
with $(m,n)\in r_{kl}$. Note that it can be only of the following five types
\begin{align*}
\begin{matrix}
0 & 0\\
0 & 0
\end{matrix}\qquad\;
\begin{matrix}
1 & 0\\
1 & 0
\end{matrix}\qquad\;
\begin{matrix}
1 & 0\\
0 & 0
\end{matrix}\qquad\;
\begin{matrix}
1 & 1\\
1 & 0
\end{matrix}\qquad\;
\begin{matrix}
1 & 1\\
1 & 1
\end{matrix}
\end{align*}
where $0$ stands for the terms with $\log_2 n<m(m+1)p'$, while $1$, for those with $\log_2 n\geq m(m+1)p'$. We will write $(m,n)\in T_i,\;i=1,...,5,$ if the corresponding quadruple is of the $i$th type. Note that if $(m,n)\in T_3$, then $(m-1,n)\in T_1$ and $(m+1,n)\in T_2,$ while if $(m,n)\in T_4$, then $(m-1,n)\in T_2$ and $(m+1,n)\in T_5$. By the construction, quadruples of the three last types with nonzero $\Delta^{11}a_{mn}$ can appear at most four times in $r_{kl}$, since any $(m,n)\in T_3\cup T_4$, as well as $(m,n)\in T_5$ with nonzero $\Delta^{11}a_{mn}$, satisfies $n\in I_j,n+1\in I_{j+1}$, for some $j$, which cannot happen twice in $[l,2l]$. %contains a part of a column which changes the type of its entries from $0$ to $1$ содержат часть столбца, в котором произошло серьёзное изменение, а значит, лежат на границе двух двоичных интервалов по $n$. При этом $[k,2k]$ может пересекать не более двух двоичных интервалов.
 If there exists a quadruple of the first type, then
\begin{align*}
\sum_{(m,n)\in T_1\cap r_{kl}}|\Delta^{11}a_{mn}|&=\sum_{(m,n)\in T_1\cap r_{kl}}\Delta^{11}a_{mn}<\sum_{m\geq k,\;n\geq l}\Delta^{11}(m^{-3-\gamma}n^{-\frac{1}{p'}})=k^{-3-\gamma}l^{-\frac{1}{p'}}\lesssim \max_{(m,n)\in r_{kl}}|a_{mn}|.
\end{align*}
%all the terms are positive and $\Delta^{11}a_{mn}$ is also positive, so the sum of $|\Delta^{11}a_{mn}|$ over $(m,n)\in T_1\cap r_{kl}$ is up to a constant less that the maximal term in $r_{kl}$. 
As for $(m,n)\in T_2\cap r_{kl}$, they all belong to a strip $[k',k'+1]\times [l,2l]$ for some $k'$. Indeed, otherwise there are $m_1$ and $m_2\geq m_1+2$ belonging to $[k,2k]$, and $n_1,n_2\in [l,2l]$ such that $(m_1,n_1),(m_2,n_2)\in T_2$. But it follows from $(m_1,n_1)\in T_2$ that $a_{m_1+1,k}$, and hence $a_{m_2,k}$, has type $0$, while $(m_2,n_2)\in T_2$ implies that $a_{m_2,2k}$, and hence $a_{m_1+1,2k}$, has type $1$. Thus, there exist two pairs of the form $(n,n+1)$ inside $[l,2l]$ such that $n\in I_j,n+1\in I_{j+1}$, for some $j$, which cannot be true. Therefore, all $(m,n)\in T_2\cap r_{kl}$ do belong to a strip $[k',k'+1]\times [l,2l]$, whence using
%there are three columns corresponding to some $m',m'+1,$ and $m'+2$, that change their type from $0$ to $1$ inside $r_{kl}$. This means that %there exists $m'\in[k,2k]$ such that 
% $\log_2 l<m'(m'+1)p'$, while $\log_2 (2l)\geq (m'+1)(m'+2)p',$ whence $m'(m'+1)p'+1>(m'+1)(m'+2)p',$ which is impossible. Therefore, using
\begin{align*}
|\Delta^{11}a_{mn}|\leq |\Delta^{01}a_{mn}|+|\Delta^{01}a_{m+1,n}|=\Delta^{01}|a_{mn}|+\Delta^{01}|a_{m+1,n}|,
%|a_{mn}-a_{m+1,n}-a_{m,n+1}+a_{m+1,n+1}|&\leq |a_{mn}-a_{m,n+1}|+|a_{m+1,n}-a_{m+1,n+1}|\\
%&=|a_{mn}|-|a_{m,n+1}|+|a_{m+1,n}|-|a_{m+1,n+1}|,
\end{align*} 
which is true as long as $(m,n)\in T_2\cap r_{kl}$, we deduce that the sum of $|\Delta^{11}a_{mn}|$ over $(m,n)\in T_2\cap r_{kl}$ is bounded above by four times the maximal $|a_{mn}|$ in $r_{kl}$. Combining the observations above, we arrive at
\begin{align*}
&\sum_{m=k}^{2k}\sum_{n=l}^{2l}|\Delta^{11}a_{mn}|\lesssim \max_{(m,n)\in r_{kl}}|a_{mn}|\lesssim |a_{2k,l}|,
\end{align*}
which proves \eqref{cond5}. 

Further, for any $q>0$,
\begin{align*}
\sum_{m,n=1}^{\infty}|a_{mn}|^q(mn)^{\frac{q}{p'}-1}&\gtrsim\sum_{m=1}^{\infty}m^{\frac{q}{p'}-1-\gamma q}\sum_{t=0}^{\infty}2^{-((m+t)^2+3m+3t)q}2^{((m+t)^2+3m+t)p'(\frac{q}{p'}-1)}2^{((m+t)^2+3m+t)p'}\\
&\gtrsim\sum_{m=1}^{\infty}m^{\frac{q}{p'}-1-\gamma q}=\infty,
\end{align*}
if we set $\gamma = 1/p'$.

Note that our sequence generates the Fourier sine (or cosine) series of an odd (or even) function $f$ that converges in the Pringsheim sense everywhere on $(0,2\pi)^2$ to $f$ according to \cite[L. 4]{D-1}. To prove this, since the sequence fulfils \eqref{cond1}, it suffices to show that the following sum is finite
\begin{align*}
\sum_{m,n=1}^{\infty}|\Delta^{11}a_{mn}|&\leq\sum_{(m,n)\in T_1}\Delta^{11}a_{mn}+\sum_{(m,n)\in T_2\cup T_5}(|\Delta^{01}a_{mn}|+|\Delta^{01}a_{m+1,n}|)\\
&+\sum_{(m,n)\in T_3\cup T_4}(|a_{mn}|+|a_{m,n+1}|+|a_{m+1,n}|+|a_{m+1,n+1}|)\\
&\lesssim 1+\sum_{(m,n)\in T_2\cup T_5}(\Delta^{01}a_{mn}+\Delta^{01}a_{m+1,n})+\sum_{m=1}^{\infty}m^{-3-\gamma}2^{-m(m+1)}\\
&\lesssim 1+\sum_{m=1}^{\infty}\sum_{t=0}^{\infty}2^{-(m+t)^2-3(m+t)}+\sum_{(m,n)\in T_2}\Delta^{01}a_{m+1,n}\\
&\lesssim 1+\sum_{m=1}^{\infty}m^{-3-\gamma}2^{-m(m-1)}<\infty.
%&\sum_{m,n=1}^{\infty}m^{-\gamma-2}n^{-\frac{1}{p'}-1}\log^{-2}n+\sum_{m=1}^{\infty}m^{-\gamma-1}(2^{m(m+1)p'})^{-\frac{1}{p'}}m^{-4}\\
%&+\sum_{m=1}^{\infty}m^{-\gamma-1}\sum_{n=1}^{2^{m(m+1)p'}}n^{-\frac{1}{p'}-1}\log^{-2}n+\sum_{m=1}^{\infty}m^{-\gamma}\sum_{t=0}^{\infty}2^{-(m+t)^2-3(m+t)}<\infty.
%\sum_{m=1}^{\infty}\frac{1}{m^{\gamma}}\sum_{t=0}^{\infty}2^{-(m+t)^2-3(m+t)}+\sum_{m=1}^{\infty}\sum_{n=1}^{\infty}\frac{1}{m^{\gamma+1}}n^{-\frac{1}{p'}-1}+\sum_{m=1}^{\infty}\frac{1}{m^{\gamma+1}}<\infty.
\end{align*}
Let us stick to the case of an odd $f$, as for cosine series the argument is exactly the same. 
%Let $P_m$ be the $m$th partial sum of the Rudin-Shapiro series and 
 Denote for $m,n\geq 1$,
%\begin{align*}
%b^l_{mn}:=\begin{cases} a_{m+l-1,n}, & \text{if}\;\; (l-1)lp'\leq \log_2 n<l(l+1)p',\\
%0, & \text{otherwise},
%\end{cases},
%\end{align*}
\begin{align*}
c_{mn}:=\begin{cases} a_{mn}, & \text{if}\;\; \log_2 n\geq m(m+1)p',\\
0, & \text{otherwise},
\end{cases},
\end{align*}
and $b_{mn}:=a_{mn}-c_{mn}$. Then
\begin{align*}
\|f\|_{L_{w(p,q)}^q}&\leq \Big\|\sum_{m,n=1}^{\infty}b_{mn}\sin mx\sin ny\Big\|_{L_{w(p,q)}^q}+\Big\|\sum_{m,n=1}^{\infty}c_{mn}\sin mx\sin ny\Big\|_{L_{w(p,q)}^q}.
\end{align*}
Note that
\begin{align*}
\sum_{m=1}^M\sum_{n=1}^N b_{mn}\sin mx\sin ny&=\sum_{m=1}^{M}\sin mx\Big(\sum_{n=1}^{N-1}\Delta^{01}b_{mn}D_n(y) +b_{mN}D_N(y)\Big)\\
&=\sum_{m=1}^{M-1}\sum_{n=1}^{N-1}\Delta^{11}b_{mn} D_m(x)D_n(y)+\sum_{n=1}^{N-1}\Delta^{01}b_{Mn} D_M(x)D_n(y)\\
&+\sum_{m=1}^{M-1}\Delta^{10}b_{mN} D_m(x)D_N(y)+b_{MN}D_M(x)D_N(y)\\
&=:\sum_{m=1}^{M-1}\sum_{n=1}^{N-1}\Delta^{11}b_{mn} D_m(x)D_n(y)+A_1+A_2+A_3.
%-\sum_{m=1}^{M}\sum_{n=1}^{N}\Delta^{11}b_{mn}D_m(x)D_n(y)=\sum_{m=1}^{M}\Delta^{01}b_{m,N+1}D_m(x)D_{N+1}(y)+\sum_{n=1}^Nb_{Mn}\sin mx\sin nx
\end{align*}
Since $\|D_k\|_{L_{w(p,q)}^q}^q\asymp \sum_{l=1}^k l^{\frac{q}{p'}-1}\asymp k^{\frac{q}{p'}}$ by Theorem \ref{sagh}, we have for $N_0:=\max(N-1,\lceil 2^{M(M+1)p'}\rceil -1)$,
\begin{align*}
\|A_1\|_{L_{w(p,q)}^q}\lesssim \sum_{n=1}^{N_0}M^{-3-\gamma}n^{-1-\frac{1}{p'}}(Mn)^{\frac{1}{p'}}+M^{-3-\gamma}N_0^{-\frac{1}{p'}}(MN_0)^{\frac{1}{p'}}\lesssim M^{-1-\gamma}\to 0
\end{align*}
as $M\to\infty$. For $M_0:=\min\{m:m(m+1)p'\geq N\}$, 
\begin{align*}
\|A_2\|_{L_{w(p,q)}^q}\lesssim \sum_{m= M_0}^{M-1}m^{-4-\gamma}N^{-\frac{1}{p'}}(mN)^{\frac{1}{p'}}+M_0^{-3-\gamma}N^{-\frac{1}{p'}}(M_0N)^{\frac{1}{p'}}\to 0
\end{align*}
as $N\to\infty$. And finally,
%(\log_2 N/p')^{1/2}-1
\begin{align*}
\|A_3\|_{L_{w(p,q)}^q}\lesssim M^{-3-\gamma}N^{-\frac{1}{p'}}(MN)^{\frac{1}{p'}}\to 0
\end{align*}
as $M\to\infty$. Thus,
\begin{align}\label{add_a}
\Big\|\sum_{m,n=1}^{\infty} b_{mn}\sin mx\sin ny\Big\|_{L_{w(p,q)}^q}=\Big\|\sum_{m,n=1}^{\infty}\Delta^{11}b_{mn} D_m(x)D_n(y)\Big\|_{L_{w(p,q)}^q}.
\end{align}
Besides,
\begin{align*}
\sum_{m=1}^M\sum_{n=1}^N c_{mn}\sin mx\sin ny&=\sum_{m=1}^{M}\sin mx\Big(\sum_{n=1}^{N-1}\Delta^{01}c_{mn}D_n(y) +c_{mN}D_N(y)\Big),
%&=\sum_{m=1}^{M}\sum_{n=1}^{N-1}\Delta^{01}c_{mn}\sin mx D_n(y) +B,
\end{align*}
where in light of the inequalities $0<g_1(n)\leq ...\leq g_{M_0}(n)$ %в силу неравенства $3(k^2+3k)> (k+1)^2+3(k+1)$, верного при всех $k\geq 1$, 
 for $M_0$ defined as above %=\min\{m:m(m+1)p'\geq \log_2 N\}-1$, we have %\eqref{ostzam}, $c_{mn}\lesssim m^{-3-\gamma}n^{-1/p'}$, следовательно,
\begin{align*}
\Big\|\sum_{m=1}^{M}c_{mN}\sin mxD_N(y)\Big\|_{L_{w(p,q)}^q}\lesssim \sum_{m=1}^{M_0} |c_{mN}|N^{\frac{1}{p'}}\leq M_0g_{M_0}(N)N^{\frac{1}{p'}}\lesssim M_0^{-2}\to 0
\end{align*}
as $N\to\infty$. Hence, 
\begin{align}\label{add_b}
\Big\|\sum_{m,n=1}^{\infty} c_{mn}\sin mx\sin ny\Big\|_{L_{w(p,q)}^q}=\Big\|\sum_{m,n=1}^{\infty}\Delta^{01}c_{mn} \sin mxD_n(y)\Big\|_{L_{w(p,q)}^q}.
\end{align}
Combining \eqref{add_a} and \eqref{add_b} we arrive at%Тогда с помощью преобразования Абеля приходим к
\begin{align*}
\|f\|_{L_{w(p,q)}^q}&\leq \Big\|\sum_{m,n=1}^{\infty}\Delta^{11}b_{mn}D_m(x)D_n(y)\Big\|_{L_{w(p,q)}^q}+\Big\|\sum_{m,n=1}^{\infty}\Delta^{01}c_{mn}\sin mx D_n(y)\Big\|_{L_{w(p,q)}^q}=:S_1+S_2.
\end{align*}
%Then, by Abel's transformation,
%\begin{align*}
%\|f\|_{L_{w(p,q)}^q}&\leq \Big\|\sum_{m,n=1}^{\infty}\Delta^{11}b_{mn}D_m(x)D_n(y)\Big\|_{L_{w(p,q)}^q}+\Big\|%\sum_{m,n=1}^{\infty}\Delta^{01}c_{mn}e^{imx}D_n(y)\Big\|_{L_{w(p,q)}^q}=:S_1+S_2.
%\end{align*}
%Since $\|D_k\|_{L_{w(p,q)}^q}^q\asymp \sum_{l=1}^k l^{\frac{q}{p'}-1}\asymp k^{\frac{q}{p'}}$, we have %for $\gamma=1/p'$,
First, for $n_{m}=\lceil 2^{m(m+1)p'}\rceil -1$, we see that $\log n_m\asymp m^2$ and $\log n_{m+1}-\log n_m\asymp m$, so
\begin{align*}
S_1&\lesssim \sum_{m=1}^{\infty}m^{\frac{1}{p'}}\Big(\sum_{n=1}^{n_m-1}\Delta^{11}(m^{-3-\gamma}n^{-\frac{1}{p'}})n^{\frac{1}{p'}} +\sum_{n=n_m}^{n_{m+1}-1}\Delta^{01}((m+1)^{-3-\gamma}n^{-\frac{1}{p'}})n^{\frac{1}{p'}}+(m^{-3-\gamma}n_m^{-\frac{1}{p'}})n_m^{\frac{1}{p'}}\Big)\\
&\lesssim\sum_{m=1}^{\infty}m^{\frac{1}{p'}}\Big(\sum_{n=1}^{n_m-1}m^{-4-\gamma}n^{-1}+\sum_{n=n_m}^{n_{m+1}-1}m^{-3-\gamma}n^{-1} +m^{-3-\gamma}\Big)\lesssim \sum_{m=1}^{\infty}m^{\frac{1}{p'}-2-\gamma}<\infty.
\end{align*}
Second, denoting $n_{mt}:=\lceil2^{((m+t)^2+3m+t)p'}\rceil -1$, using $c_{mn}=(-1)^{\delta_m}|c_{mn}|$ and the fact that $\Delta^{01}c_{mn}\neq 0$ only if $n=n_{mt}$ for $t\geq -1$, we get for $q\geq p$, %for an appropriate choice of $\delta_m$, by the Khintchine and Minkowski inequalities (recall also that $q\geq p$), we get
\begin{align}\label{befkh}
S_2^q&=\Big\|\sum_{m=1}^{\infty}(-1)^{\delta_m}\sin mx\sum_{t=-1}^{\infty}\Delta^{01}|c_{m,n_{mt}}|D_{n_{mt}}(y)\Big\|_{L_{w(p,q)}^q}^q\nonumber\\
%&\leq \Big\|\sum_{m=1}^{\infty}(-1)^{\delta_m}\sin mx\sum_{t=-1}^{\infty}\Delta^{01}|c_{m,n_{mt}}|D_{n_{mt}}(y)\Big\|_{L_{w(p,q)}^q}^q\nonumber\\
&=\int\limits_{-\pi}^{\pi}|y|^{\frac{q}{p}-1}\int\limits_{-\pi}^{\pi}|x|^{\frac{q}{p}-1}\Big|\sum_{m=1}^{\infty}(-1)^{\delta_m}\sin mx\sum_{t=-1}^{\infty}\Delta^{01}|c_{m,n_{mt}}|D_{n_{mt}}(y)\Big|^qdx dy\nonumber\\
&\leq\int\limits_{-\pi}^{\pi}|y|^{\frac{q}{p}-1}\int\limits_{-\pi}^{\pi}\Big|\sum_{m=1}^{\infty}(-1)^{\delta_m}\sin mx\sum_{t=-1}^{\infty}\Delta^{01}|c_{m,n_{mt}}|D_{n_{mt}}(y)\Big|^qdx dy.
\end{align}
By the Khintchine inequality (see e.g. \cite[Rem. 1.4]{A}) 
%and Fatou's lemma,
 we have for any real sequence $\{s_k\}\in l_2$ and the system of Rademacher functions $\{r_n(t)\}$ that
\begin{align*}
\int\limits_0^1\Big|\sum_{k=1}^{\infty}s_kr_k(t)\Big|^q\asymp_q \Big(\sum_{k=1}^{\infty}s_k^2\Big)^{\frac{q}{2}},
\end{align*}
whence
\begin{align}\label{hh}
\int\limits_0^1\int\limits_{-\pi}^{\pi}&\Big|\sum_{m=1}^{\infty}r_m(t)\sin mx\sum_{t=-1}^{\infty}\Delta^{01}|c_{m,n_{mt}}|D_{n_{mt}}(y)\Big|^qdxdt\nonumber\\
&\lesssim\int\limits_0^1\Big|\sum_{m=1}^{\infty}r_m(t)\sum_{t=-1}^{\infty}\Delta^{01}|c_{m,n_{mt}}|D_{n_{mt}}(y)\Big|^qdt\lesssim \Big(\sum_{m=1}^{\infty}\Big(\sum_{t=-1}^{\infty}\Big|\Delta^{01}|c_{m,n_{mt}}| D_{n_{mt}}(y)\Big|\Big)^2\Big)^{\frac{q}{2}},
\end{align}
whenever the series on the right-hand side converges. Note that by the Minkowski inequality and the fact that $\|D_{n_{mt}}\|_{L_{w(p,q)}^q}\asymp 2^{((m+t)^2+3m+t)p'\frac{1}{p'}}$, we have
\begin{align}\label{hh2}
%\int\limits_0^1\int\limits_{-\pi}^{\pi}|y|^{\frac{q}{p}-1}&\int\limits_{-\pi}^{\pi}\Big|\sum_{m=1}^{\infty}r_m(t)\sin mx\sum_{t=-1}^{\infty}\Delta^{01}|c_{m,n_{mt}}|D_{n_{mt}}(y)\Big|^qdx dy dt\nonumber\\
\int\limits_{-\pi}^{\pi}|y|^{\frac{q}{p}-1}\Big(&\sum_{m=1}^{\infty}\Big(\sum_{t=-1}^{\infty}\Big|\Delta^{01}|c_{m,n_{mt}}| D_{n_{mt}}(y)\Big|\Big)^2\Big)^{\frac{q}{2}}dy\nonumber\\
\asymp \Big\|&\sum_{m=1}^{\infty}\Big(\sum_{t=-1}^{\infty}\Big|\Delta^{01}|c_{m,n_{mt}}|D_{n_{mt}}(y)\Big|\Big)^2\Big\|_{L_{w(p/2,q/2)}^{q/2}}^{\frac{q}{2}}\nonumber\\
\lesssim \Big(&\sum_{m=1}^{\infty}\Big\|\sum_{t=-1}^{\infty}\Big|\Delta^{01}|c_{m,n_{mt}}|D_{n_{mt}}(y)\Big|\Big\|_{L_{w(p,q)}^q}^2\Big)^{\frac{q}{2}}\nonumber\\
\lesssim\Big(&\sum_{m=1}^{\infty}m^{-2\gamma}\Big(2^{-m^2-3m}2^{m(m+1)}+\sum_{t=0}^{\infty}2^{-((m+t)^2+3(m+t))}2^{((m+t)^2+3m+t)}\Big)^2\Big)^{\frac{q}{2}}\nonumber\\
\lesssim \Big(&\sum_{m=1}^{\infty}m^{-\frac{2}{p'}}\Big)^{\frac{q}{2}}<\infty.
\end{align}
Thus, by \eqref{hh} and \eqref{hh2}, for almost all $t$, the sum $\sum_{m=1}^{\infty}r_m(t)\sin mx\sum_{t=-1}^{\infty}\Delta^{01}|c_{m,n_{mt}}|D_{n_{mt}}(y)$ converges for almost all $y$ uniformly in $x$, and moreover, \eqref{hh} and \eqref{hh2} imply that
\begin{align*}
\int\limits_{-\pi}^{\pi}|y|^{\frac{q}{p}-1}&\int\limits_{-\pi}^{\pi}\Big|\sum_{m=1}^{\infty}r_m(t)\sin mx\sum_{t=-1}^{\infty}\Delta^{01}|c_{m,n_{mt}}|D_{n_{mt}}(y)\Big|^qdx dy<\infty
\end{align*}
for almost all $t$ (denote this set by $E\subset (0,1)$). Taking any $t_0\in E\setminus\{k2^{-l}\}_{k,l\in\mathbb{N},k<2^l},$ so that $r_m(t_0)=\pm 1$ for all $m$, and setting $\{\delta_m\}$ according to the equality $(-1)^{\delta_m}=r_m(t_0)$, we obtain in light of \eqref{befkh} that $S_2<\infty$.
\end{proof}

\begin{acknowledgements} I am grateful to Mikhail Dyachenko and Sergey Tikhonov for acquainting me with this topic and for the many issues I have learnt from our conversations. I also thank Erlan Nursultanov, Miquel Saucedo, and the anonymous referees for their valuable comments and suggestions.
\end{acknowledgements}


\begin{thebibliography}{24}
\bibitem{AW}
R. Askey, S. Wainger, {\it Integrability theorems for Fourier series}, Duke Math. J. 33  (1966), 223--228.

\bibitem{A}
S. Astashkin, {\it The Rademacher system in function spaces}, Birkhauser, 2020.

\bibitem{BL}
J. Bergh, J. L\"ofstr\"om, {\it Interpolation spaces. An introduction}, Springer, New York, 1976.

%\bibitem{B}
%R. P. Boas Jr., {\it Integrability theorems for trigonometric transforms}, Springer, New York, 1967.

\bibitem{DoT}
O. Dom\'{i}nguez, S. Tikhonov, {\it Function spaces of logarithmic smoothness: embeddings and characterizations}, to appear in Memoirs Amer. Math. Soc., arxiv: 1811.06399.

\bibitem{DS}
K. Duzinkiewicz, B. Szal, {\it On weighted integrability of double sine series}, J. Math. Anal. Appl. 472  (2019), 1581--1603.

\bibitem{D-1}
M.I. Dyachenko, {\it On the convergence of double trigonometric series and Fourier series with monotone coefficients}, Mat. Sb. 129 (1)  (1986), 55--72; Math. USSR Sb. 57  (1987), 57--75 (in English).

\bibitem{D-2}
M.I. Dyachenko, {\it Multiple trigonometric series with lexicographically monotone coefficients}, Anal. Math. 16(3)  (1990), 173--190.

\bibitem{D-3}
M.I. Dyachenko, {\it Norms of Dirichlet kernels and some other trigonometric polynomials in $L_p$-spaces}, Sb. Math. 78(2)  (1994), 267--282.

\bibitem{DMT}
M. Dyachenko, A. Mukanov, S. Tikhonov, {\it Hardy-Littlewood theorems for trigonometric series with general monotone coefficients}, Studia Math. 250(3)  (2019), 217--234.

\bibitem{DT-0}
M. Dyachenko, S. Tikhonov, {\it Convergence of trigonometric series with general monotone coefficients},  C. R. Acad. Sci. Paris Ser. I 345  (2007), 123--126.

\bibitem{DT-1}
M. Dyachenko, S. Tikhonov, {\it A Hardy-Littlewood theorem for multiple series}, J. Math. Anal. Appl. 339  (2008) 503--510.

\bibitem{DT-11}
M. Dyachenko, S. Tikhonov, {\it General monotone sequences and convergence of trigonometric series}, Topics in Classical Analysis and Applications in Honor of Daniel Waterman, 2008.

\bibitem{F}
C. Fefferman, {\it On the convergence of multiple Fourier series}, Bull. Amer. Math. Soc. 77  (1971), 744--745.

\bibitem{FTZ}
L. Feng, V. Totik, S. P. Zhou, {\it Trigonometric Series with a Generalized Monotonicity Condition}, Acta Math. Sin. (Engl. Ser.) 30(8)  (2014), 1289--1296.

\bibitem{GT}
D. Gorbachev, S. Tikhonov, {\it Moduli of smoothness and growth properties of Fourier transforms: Two-sided estimates}, Jour. Appr. Theory 164(9)  (2012), 1283--1312.

\bibitem{Ha}
G. H. Hardy, {\it On double Fourier series, and especially those which represent the double zeta-function with real and incommensurable parameters}, Quart. J. Math. 37(1)  (1906), 53--79.
%(1)

\bibitem{Haa}
U. Haagerup, {\it The best constants in the Khintchine inequality}, Studia Math. 70(3)  (1981), 231--283.

\bibitem{HL}
G. H. Hardy, J. E. Littlewood, {Notes on the theory of series (XIII): Some new properties of Fourier constants}, J. London Math. Soc. 6  (1931), 3--9.

%\bibitem{H}
%P. Heywood, {\it On the integrability of functions defined by trigonometric series}, Quart. J. Math. 5 (1954), 71--76.

\bibitem{K}
J. M. Krause, {\it Fouriersche Reihen mit zwei ver\"anderlichen Gr\"ossen}, Ber. Verh. K\"onigl. S\"achs. Ges. Wiss. Leipzig 55  (1903), 164--197.

\bibitem{KP}
A. Kufner, L. E. Persson, {\it Weighted inequalities of Hardy type}, World Scientific,  Singapore, 2003.
%The American Mathematical Monthly, New Jersey, London, Hong Kong, 

%\bibitem{KS}
%B. S. Kashin, A. A. Saakyan, {\it Orthogonal series}, AMS, 2005.

%\bibitem{L}
%L. Leindler, {\it Generalization of inequalities of Hardy and Littlewood}, Acta Sci. Math. (Szeged) 31  (1970), 279--285.

\bibitem{LT}
E. Liflyand, S. Tikhonov, {\it A concept of general monotonicity and applications}, Math. Nachrichten 284  (2011), 1083--1098.

\bibitem{Ma}
L. Maligranda, {\it On interpolation of nonlinear operators}, Comment. Math.
Prace. Mat. 28(2)  (1989), 253--275.
%Annales Societatis Mathematicae Polonae. Series 1: Commentationes Mathematicae - Prace Matematyczne

\bibitem{M}
F. M\'oricz, {\it On double cosine, sine, and Walsh series with monotone coefficients}, Proc. Amer. Math. Soc. 109(2)  (1990), 417--425.

\bibitem{N-1}
E. D. Nursultanov, {\it Net spaces and inequalities of Hardy-Littlewood type}, Sb. Math. 189(3)  (1998), 399--419.  

\bibitem{N-2}
E. D. Nursultanov, {\it On the coefficients of multiple Fourier series in $L_p$-spaces}, Izv. Math. 64(1)  (2000), 93--120.

\bibitem{P}
R. Paley, {\it Some theorems on orthogonal functions}, St. M. 3  (1931), 226--238.

%\bibitem{R}
%W. Rudin, {\it Some theorems on Fourier coefficients}, Proc. Amer. Math. Soc. 10  (1959), 855--859.

\bibitem{S}
Y. Sagher, {\it Integrability conditions for the Fourier transform}, J. Math. Anal. Appl. 54  (1976), 151--156.

\bibitem{T-1}
S. Tikhonov, {\it Trigonometric series with general monotone coefficients}, J. Math. Anal. Appl. 326  (2007) 721--735.

\bibitem{T-2}
S. Tikhonov, {\it Best approximation and moduli of smoothness: Computation and equivalence theorems}, J. Approx. Theory 153  (2008), 19--39.

\bibitem{W-1}
F. Weisz, {\it Martingale Hardy spaces and their applications in Fourier analysis}, Lecture Notes in Math. vol. 1568, Springer, Berlin, Heidelberg, 1994.
%New York: Springer, Berlin, 1994.

\bibitem{W-2}
F. Weisz, {\it Inequalities relative to two-parameter Vilenkin-Fourier coefficients}, Studia Math. 99  (1991), 221--233.

\bibitem{YZZ-1} 
D. S. Yu, P. Zhou and S. P. Zhou, {\it On $L_p$ integrability and convergence of trigonometric series}, Studia Math. 182  (2007), 215--226.

\bibitem{YZZ-2}
D. Yu, P. Zhou, S. Zhou, {\it Mean bounded variation condition and applications in double trigonometric series}, Anal. Math. 38  (2012), 83--104.

\end{thebibliography}
\end{document}